\documentclass[12pt]{article}

\usepackage[english]{babel}
\usepackage{amsmath,amsthm}
\usepackage{amsfonts}
\usepackage{graphicx}
\usepackage{epsfig}
\usepackage{epstopdf}
\usepackage{url}
\usepackage{amssymb}
\usepackage{color}
\usepackage[makeroom]{cancel}
\usepackage{tikz-cd}
\usepackage{tikz}
\usetikzlibrary{matrix,calc}
\usepackage[titletoc]{appendix}
\usepackage{multirow}
\usepackage{tabularx}
\usepackage{bigdelim}
\usepackage{blkarray}
\usepackage[normalem]{ulem}

\usepackage[%
breaklinks%
,colorlinks%
,linkcolor=red% internal document links
,anchorcolor=blue%
,pagecolor=blue%
,citecolor=blue%
,bookmarks=false%
]{hyperref}

\usepackage[nameinlink,capitalise]{cleveref}

\newlength\tindent
\definecolor{delim}{RGB}{20,105,176}
\definecolor{numb}{RGB}{106, 109, 32}
\definecolor{string}{rgb}{0.64,0.08,0.08}

\newcommand{\CI}{\mathrel{\perp\mspace{-10mu}\perp}}

\newtheorem*{theorem*}{Theorem}
\newtheorem{theorem}{Theorem}
\newtheorem{proposition}[theorem]{Proposition}
\newtheorem{lemma}[theorem]{Lemma}

\newtheorem*{corollary*}{Corollary}
\newtheorem{corollary}[theorem]{Corollary}

\theoremstyle{definition}

\newtheorem{remark}[theorem]{Remark}

\newtheorem{example}[theorem]{Example}

\def\P{{\mathbb{P}}}

\def\C{{\mathcal{C}}}

\def\V{\mathcal{V}_X}
\def\F{\mathcal{F}}
\def\Mc{\mathcal{M}_\mathcal{C}}
\def\Vc{\mathcal{V}_{X,\mathcal{C}}}
\def\para{\vspace{1.5mm}}

\def\C{{\mathbb C}}
\def\Z{\mathbb{Z}}

\def\O{\mathcal{O}}

\usepackage{mathrsfs}
\usepackage{algorithm}
\usepackage{algorithmic}

\title{\bf Nash Conditional Independence
Curve}
\author{Irem Portakal and Javier Sendra--Arranz}
\date{}
\providecommand{\keywords}[1]{\textbf{\textit{Keywords---}} #1}

\Crefname{page}{page}{page}

\begin{document}
\maketitle

\begin{abstract}
\noindent We study the Spohn conditional independence (CI) variety $C_X$ of an $n$-player game $X$ for undirected graphical models on $n$ binary random variables consisting of one edge. For a generic game, we show that $C_X$ is a smooth irreducible complete intersection curve (Nash conditional independence curve) in the Segre variety $(\P^{1})^{n-2} \times \P^3$ and we give an explicit formula for its degree and genus. We prove two universality theorems for $C_X$: The product of any affine real algebraic variety with the real line or any affine real algebraic variety in $\mathbb{R}^m$ defined by at most $m-1$ polynomials is isomorphic to an affine open subset of $C_X$ for some game $X$.
\end{abstract}
\para

\section{Introduction}

The classical concept of Nash equilibria focuses on the case where each player acts independently, without communication and collaboration with the other players \cite{nash51}. In this equilibrium, no player can increase their expected payoffs by changing their strategy while assuming the other players have fixed strategies. Alongside Nash equilibria, the well-known concept of correlated equilibria \cite{aumann1} has also been found to be non-optimal for certain games. Spohn argues this phenomenon \cite{spohn2003} on the classical example of the prisoners' dilemma in which both Nash and correlated equilibrium recommend mutual betrayal, leading to harsh sentences for both prisoners. The most beneficial or Pareto optimal strategy is mutual non-betrayal, leading to light sentences for both. In an alternative perspective, Spohn addressed the scenario involving dependencies between players and proposed in \cite{spohn2003} the concept of {\em dependency equilibria}, where the players simultaneously maximize their conditional expected payoffs. In particular, he shows that mutual non-betrayal is a dependency equilibrium.\\ 

The concepts of Nash and dependency equilibria can be modeled in terms of {\em undirected graphical models} \cite[Chapter 6]{BI}. The vertices of the underlying graph of the graphical model represent the players of the game and the dependencies of the choices of the players are depicted with an edge in the graph. Associated to the undirected graphical model, we consider the global Markov properties which is a set of certain conditional independence statements that is consistent with the underlying graph. The probability distributions that satisfy the conditional independence statements can be represented as a system of homogeneous quadratic constraints, which gives rise to the conditional independence model (\cite[Proposition 8.1]{CBMS}, \cite[Proposition 4.1.6]{Sul}). This graphical modeling approach incorporates both the concepts of Nash and dependency equilibria: In the case of Nash equilibria of an $n$-player game, the graphical model consists of $n$ isolated vertices whereas for the dependency equilibria, one obtains a complete graph on $n$ vertices. The totally mixed dependency equilibria can be defined by the $2\times 2$ minors of certain matrices of linear forms i.e.\ the \emph{Spohn variety}, restricted to the open probability simplex. The algebro-geometric features of the Spohn variety are studied in \cite{BI} and the rationale behind our interest in totally mixed equilibria is explained in Section~\ref{sec: 2.1}. The \emph{Spohn conditional independence (CI) variety} is obtained by intersecting the Spohn variety with the conditional independence model. Then, the \emph{set of totally mixed conditional independence (CI) equilibria} is defined as the intersection of the Spohn CI variety with the open probability simplex. Thus, the concept of CI equilibria deals with all possible dependencies between the players of the game and bridges the gap between the totally mixed Nash and dependency equilibria. In particular, the set of totally mixed CI equilibria is contained in the totally mixed dependency equilibria and it contains the set of totally mixed Nash equilibria.\\

The next natural question is to study the case where the graphical model consists of only one-edge, which is a very close case to Nash equilibria except for the fact that dependencies only on 2 players are imposed. In this paper, we focus on the case where each player has binary choices. We investigate the properties of the set of totally mixed CI equilibria from the analysis 
of the algebro-geometric properties of the Spohn CI variety, which generically turns out to be a smooth irreducible curve. We call this curve the \emph{Nash conditional independence (CI) curve}. The study of the Spohn CI variety pertains to \cite[Conjecture 6.3]{BI}, where Bayesian networks were considered instead of undirected graphical models. Since the associated conditional independence models are the same for Bayesian networks and undirected graphical models with one edge, our results prove the part of the conjecture for Bayesian networks with one-edge.\\

This paper is the first attempt at a kind of unification of non-cooperative and cooperative game theory mentioned in \cite[Section 6.3]{spohn2003} from the perspective of algebraic statistics as well as with methods from computational algebraic geometry. The main contributions and the structure of the paper are as follows: In Section~\ref{Sec 2} we  recall the  definitions from \cite{BI} of dependency equilibria, Spohn (CI) variety, and totally mixed CI equilibria. We prove, for generic payoff tables, that the Spohn CI variety for {\em one-edge undirected graphical models} is a curve which is the complete intersection of $n$ divisors in the Segre variety $\left(\P^1\right)^{n-2}\times\P^3$, of multi-degree 
$$(0,1,\ldots,1),\ldots, (1,\ldots,1,\underset{(i)}{0},1,\ldots,1), \ldots,(1,\ldots,1,0,1,1),(1,\ldots,1,2), (1,\ldots,1,2)$$ respectively
(Proposition~\ref{prop:dim 1} and Corollary \ref{co:gen com int}). In {\tt Macaulay2} \cite{M2}, many examples were computed using the {\tt GraphicalModels} package. These lead in Section~\ref{Section 3} to the explicit formulas for the degree and the genus of the Nash CI curve (Lemma~\ref{lem: degree of almost Nash curve} and Corollary~\ref{cor: genus of the generic curve}). We also prove that the Nash CI curve is connected (Lemma~\ref{lem: connected}). Lastly, Section~\ref{Sec 4} focuses on the analysis of  different notions of the universality of the Spohn CI variety for one-edge undirected graphical models. The universality of Nash equilibria was investigated by Datta. It is proven that every real algebraic variety is isomorphic to the set of totally mixed Nash equilibria of some three-player game, and also to the set of totally mixed Nash equilibria of an $n$-player game in which each player has binary choices (\cite[Theorem 1]{datta}). A similar universality concept appears for correlated equilibria where it is proven that any convex polytope in $\mathbb{R}^n$ can be realized as the correlated equilibrium payoffs of an $n$-player game \cite[Proposition 1]{correlatedeqpayoff}. We analyze the universality from the point of view of divisors as a first approach. We study whether any set of divisors with the same multi-degree as the ones mentioned above arises from 
the Spohn CI variety. In Proposition~\ref{cor: universality of divisors}, we conclude that this notion of universality does not hold for the Spohn CI variety. Following this fact, we investigate the base locus of linear systems of these divisors and conclude in Theorem~\ref{thm: irreducible Nash CI} that the Nash CI curve for generic games is smooth and irreducible. 
As a consequence, we deduce that for generic games for which there exists a totally mixed CI equilibrium, the set of totally mixed CI equilibria is a smooth manifold of dimension $1$.
The second approach follows the spirit of Datta's universality theorem for Nash equilibria in \cite{datta}. We prove two universality theorems for Spohn CI varieties for one-edge undirected graphical models.\\

\begin{corollary*}[{Corollary~\ref{co:affine univ 1}}]
      Let $S$ be any affine real algebraic variety. There exists a game with binary choices such that the affine open subset of Spohn CI variety $C_X$ for one-edge undirected graphical models is isomorphic to $S \times \mathbb{R}^1$. 
\end{corollary*}

\vspace{0.4cm}
\begin{theorem*}[{Theorem~\ref{thm: universality of Spohn CI}}]
        Any affine real algebraic variety $S \subseteq \mathbb{R}^n$ defined by $m$ polynomials $G_1, \ldots , G_m$ in $\mathbb{R}[x_1, \ldots, x_n]$ with $m<n$ is isomorphic to an affine open subset of Spohn CI variety $C_X$ for one-edge undirected graphical models of an $N$-player game with binary choices.  
\end{theorem*}

\section{Nash and dependency equilibria meet graphical models}\label{Sec 2}
\subsection{Dependency equilibria and Spohn CI variety}\label{sec: 2.1}
In this section, we briefly introduce the Spohn conditional independence (CI) variety, the dependency equilibria and the conditional independence (CI) equilibria. We work in the setting of normal form games $X$ with $n$ players. For $i\in[n]$, the $i$th player can select from $[d_i]$ strategies and its payoff table is a tensor $X^{(i)}$ of format $d_1 \times \cdots \times d_n$ with real entries. Here, $X^{(i)}_{j_1 \ldots j_n}$ represents the payoff that player $i$ receives when player $1$ chooses strategy $j_1$, player $2$ chooses strategy $j_2$, etc. Let $V =\mathbb{R}^{d_1} \times \cdots \times \mathbb{R}^{d_n}$ be the real vector space of all tensors, and $\mathbb{P}(V)$ the corresponding projective space. Let $p_{j_1\ldots j_n}$ be the coordinates of $\P(V)$. The entry $p_{j_1\ldots j_n}$ is the probability that the first player chooses the strategy $j_1$, the second player $j_2$, etc. We study the case of totally mixed equilibria points, i.e.\ positive real points of $\P(V)$ in the open probability simplex $\Delta:=\Delta_{d_1 \ldots d_{n}-1}^{\circ}$.  The reason why we consider the positive real points is primarily rooted in the technical aspects inherent to the definition of dependency equilibrium. In particular, as we consider the conditional expected payoffs, it is essential to avoid the cases where the denominator in the conditional probabilities becomes zero. An additional reason is our investigation of universality for Spohn CI varieties in Section~\ref{Sec 4}, similar to Datta's universality result for totally mixed Nash equilibria \cite{datta}. The extension of the definition of dependency equilibrium to include the boundary of the probability simplex is explored in the ongoing work~\cite{SRR23}.

\para 

The \textit{expected payoff} of the $i$th player is defined as the following dot product:
\[
PX^{(i)} = \displaystyle\sum_{j_1=1}^{d_1}\cdots \sum_{j_n=1}^{d_n}X^{(i)}_{j_1\cdots j_n}p_{j_1\cdots j_n}.
\]

The classical theory of Nash equilibria studies the (mixed) strategies (or tensors) $P~\in~\Delta_{d_1 \cdots d_n -1}$ where no player can increase their expected payoff by changing their mixed strategy while assuming the other players have fixed mixed strategies. The concept of dependency equilibrium was introduced by a prominent philosopher Wolfgang Spohn in \cite{spohn2003}. To incorporate causal dependencies between the players, one considers conditional probabilities. The \textit{conditional expected payoff} of the $i$th player is then defined as the expected payoff conditioned on player $i$ having fixed pure strategy $k \in [d_i]$ 
\[
\displaystyle\sum_{j_1=1}^{d_1}\cdots\widehat{\sum_{j_i=1}^{d_i}}\cdots \sum_{j_n=1}^{d_n}X^{(i)}_{j_1\cdots k \cdots  j_n}\frac{p_{j_1\cdots k \cdots j_n}}{p_{+\cdots+k+\cdots+}},
\]
where $$p_{+\cdots+k+\cdots+}=
\displaystyle\sum_{j_1=1}^{d_1}\cdots\widehat{\sum_{j_i=1}^{d_i}}\cdots \sum_{j_n=1}^{d_n}
p_{j_1\cdots k \cdots j_n}.
$$
 We say that a tensor $P\in\Delta$ is a (totally mixed) \textit{dependency equilibrium} of the game if the conditional expected payoff of each player $i$ does not depend on the strategy $k$ of player $i$.  Portakal and Sturmfels \cite{BI} established the algebro-geometric foundations for the theory of dependency equilibrium by studying an algebraic variety in complex projective space obtained by relaxing the reality and positivity constraints. 
This is called the {\emph{Spohn variety}} $\mathcal{V}_X$ of the game $X$, which is defined as the $2\times 2$ minors of the matrices following $d_i \times 2$ matrices of linear forms $M_1, \ldots, M_n$:
 \begin{equation}\label{eq: Spohn matrices}
M_i \,=\, M_i(P) \,\,:= \,\,\,\begin{bmatrix}
\vdots & \vdots \\
\,\,p_{+\cdots + k + \cdots+}\, &\,\, \displaystyle\sum_{j_1 = 1}^{d_1} \cdots \widehat{\displaystyle\sum_{j_{i} = 1}^{d_{i}}} \cdots \displaystyle\sum_{j_n = 1}^{d_n} X^{(i)}_{j_1 \cdots  k  \cdots j_n} p_{j_1 \cdots  k  \cdots j_n}\, \\
\vdots & \vdots 
\end{bmatrix} \! .
\end{equation}
The set of dependency equilibria of the game $X$ is the intersection $\mathcal{V}_X \cap \Delta$.  By \cite[Theorem 6]{BI}, for generic payoff tables $X^{(1)},\ldots, X^{(n)}$, the Spohn variety is irreducible of codimension $d_1+\cdots +d_n-n$ and degree $d_1\cdots d_n$. We say that  $P\in\mathcal{V}_X$ is a \textit{Nash point} if $P$ is a tensor of rank one. Thus, the set of Nash points is the intersection of $\mathcal{V}_X$ with the Segre variety $\P^{d_1-1}\times\cdots\times\P^{d_n-1}$. Then, the set of totally mixed Nash equilibria is 
the intersection
\begin{equation}\label{eq:int equi}
\mathcal{V}_X\cap\left( \P^{d_1-1}\times\cdots\times\P^{d_n-1}\right) \cap \Delta.
\end{equation}

Note that the Spohn variety is generically high dimensional. Since the set of all dependency equilibria is $\V\cap\Delta$, it is either empty or it has the same dimension as $\V$. In order to drop the dimension and investigate different cases of dependencies between the players of the game, we focus on the intersection of $\V$ with statistical models arising from conditional independence statements in $\Delta$. To this aim, we model the dependencies between the players in terms of {\em graphical models} whose underlying graph is an undirected graph on $n$ vertices. We consider the $n$ players as discrete random variables with state spaces $[d_1], \ldots, [d_n]$. An edge between two vertices corresponds to the dependency of the choices of the associated players in the game. We consider the {\em  global Markov property} associated to the undirected graphical model. This consists of all conditional independence statements \cite[Definition 4.1.2]{Sul} of form $A \CI B \ | \ C$ such that $C$ separates $A$ and $B$ where $A,B,C$ are pairwise disjoint subsets of $[n]$ \cite[Definition 13.1.1]{Sul}. Let $\mathcal{C}$ be the collection of these CI statements. Each CI statement translates into a system of homogeneous quadratic constraints in the tensor entries $p_{j_1, \ldots, j_n}$ \cite[Proposition 4.1.6]{Sul}. We call the projective variety $\mathcal{M}_\mathcal{C}\subseteq\mathbb{P}(V)$ defined by these quadrics arising from all conditional independence statements the {\em conditional independence model}. Here we assume that components lying in the hyperplanes
$\{p_{j_1j_2 \cdots j_n}=0\}$ and $\{p_{++ \cdots +} = 0\}$ have been removed. The intersection of $\mathcal{M}_\mathcal{C}$ with the open simplex consists of the set of all the probabilities satisfying  the conditional independence statements in $\mathcal{C}$.
The {\em Spohn conditional independence (CI) variety} is defined as the intersection
of the Spohn variety with the CI model:
\begin{equation}
\label{eq:spohnCI} \mathcal{V}_{X,\mathcal{C}} \,\, = \,\,
\mathcal{V}_X \, \cap \, \mathcal{M}_\mathcal{C}. 
\end{equation}
We again assume that components lying in the special hyperplanes above have been removed. We call the intersection $\mathcal{V}_{X,\mathcal{C}} \cap \Delta$ the set of {\em totally mixed conditional independence (CI) equilibria}.
In other words, the set of totally mixed CI equilibria consists of all the dependency equilibria satisfying the CI statements of the graphical model, which translates into the dependencies between players. Note that since we consider the totally mixed equilibria, by \cite[Theorem 13.1.4]{Sul}, the specific choice of global Markov property for the graphical model becomes inconsequential.
An important first observation is that totally mixed Nash equilibria and dependency equilibria can be regarded as special cases of totally mixed CI equilibria.

\begin{example}\label{ex: Nash+dependency}
If the graphical model is the complete graph, then
the associated collection of CI statements $\mathcal{C}$ is empty and $\mathcal{M}_{\mathcal{C}} = \P(V)$. The Spohn CI variety coincides with the Spohn variety and the set of dependency equilibria is the set of totally mixed CI equilibria. On the other hand, if the graphical model has no edges, i.e.\ it consists of isolated vertices, then $\mathcal{C}$ consists of all possible CI statements and $\mathcal{M}_\mathcal{C}$ equals to the Segre variety $\P^{d_1-1}\times\cdots\times\P^{d_n-1}$. By equation (\ref{eq:int equi}), the set of totally mixed Nash equilibria coincides with the set of totally mixed CI equilibria. In particular, the set of CI equilibria coming from a collection of CI statements $\mathcal{C}$ is always contained in the set of dependency equilibria and it contains the set of totally mixed Nash equilibria. 
 \end{example}

\para
\subsection{Nash conditional independence curve}
As seen in Example~\ref{ex: Nash+dependency}, the concept of  CI equilibria places Nash and dependency equilibria at opposite ends of the spectrum within the context of graphical models. In this paper, we embark on exploring the intermediate scenarios that exist between these two extremes. In the sequel, we study the Spohn CI variety of one-edge undirected graphical models for a game with binary choices i.e.\ $d_1=\cdots=d_n=2$. The Spohn variety and $\Mc$ are projective subvarieties in the projective space $\P^{2^n-1}$. In particular, $\Mc$ is the Segre variety $\left(\P^1\right)^{n-2}\times \P^3\subset \P^{2^n-1}$. Throughout this paper, $C_X$ denotes the Spohn CI variety for one-edge undirected graphical models on $n$ binary random variables. The goal of this paper is to study the set of totally mixed CI equilibria for one-edge undirected graphical models through the algebro-geometric analysis of $C_X$.
\para

\begin{example}[3-player game]
The Spohn variety for a generic 3-player game with binary choices is a fourfold of degree 8 in $\P^7$. Moreover, by intersecting $\mathcal{V}_X$ with the Segre variety $\P^1 \times \P^{1} \times \P^{1}$ we obtain the totally mixed Nash equilibria as in \cite[Example 3]{BI}. Now, let us consider a non-generic game:  El Farol bar problem for three players. Three friends would like to meet in a bar which has limited seats. Each player has two choices: Go=1 or Stay=2. The people who go to the bar are rewarded when there are few of them (in our case two) and forfeited if the bar is overcrowded, i.e.\ if all three go to the bar. Those who choose to stay at home are rewarded when the bar is overcrowded and forfeited if they are the only one who showed up. We represent this in terms of $2 \times 2 \times 2$ payoff tensor for each player.

\begin{equation}
\begin{matrix} 
   X^{(1)} = \\ X^{(2)} =  \\ X^{(3)}= 
\end{matrix}\,
\bordermatrix{     
            & 111        & 121      &   112   & 122     & 211     & 221     & 212     & 222 \cr
            & -1 & 2  & 2 & 0 & 1 & 1 & 1 &1    \cr
            & -1 & 1  & 2 & 1 & 2 & 1 & 0 &1    \cr
            & -1 & 2  & 1 & 1 & 2 & 0 & 1 &1    \cr
}.
\end{equation}
The Spohn variety $\mathcal{V}_X$ is a 5-dimensional reducible variety of degree 8 in $\P^7$. In particular, we obtain that the Spohn CI variety $C_X$ is a reducible surface of degree 8 in the Segre variety $\P^1 \times \P^3$.     
\end{example}
\para \para

For a generic game, we obtain by \cite[Theorem 6]{BI} that 
$$\mathrm{codim}_{\P^{2^n-1}} \V=d_1+\cdots+d_n-n= 2n-n= n.$$
Since the dimension of $\Mc$ is $n+1$, we expect that for generic payoff tables, the Spohn CI variety is a curve (see \cite[Conjecture 23]{BI}).

\para

In order to define the ideal of $\Vc$, our first step is to evaluate the parametrization of the Segre variety $\Mc$ in the equations of $\V$, i.e.\ in the determinant of $M_i$ for every $i \in [n]$. We utilize the parametrization of $\Mc$ given by the Segre embedding 
\[
p_{j_1\ldots j_n}=\sigma^{(1)}_{j_1}\cdots \sigma^{(n-2)}_{j_{n-2}}\tau_{j_{n-1},j_n}
\]
for $j_1,\ldots,j_n\in[2]$. Evaluating the determinant of $M_i$, for $i\leq n-2$, at this parametrization, we obtain that the determinant of $M_i$ is the product of $\sigma^{(i)}_{1}\sigma^{(i)}_{2} \left(
\sum_{j_1,\ldots,\widehat{j_i},\ldots,j_n}\sigma^{(1)}_{j_1}\cdots \widehat{\sigma^{(i)}_{j_{i}}}\cdots \sigma^{(n-2)}_{j_{n-2}}\tau_{j_{n-1},j_n}
\right)$ with
\[
\begin{array}{ll}
F_i:=
\displaystyle\sum_{j_1,\ldots,\widehat{j_i},\ldots,j_n}
\left(X^{(i)}_{j_1\cdots2\cdots j_n}- X^{(i)}_{j_1\cdots 1 \cdots j_n} \right)
\sigma^{(1)}_{j_1}\cdots \widehat{\sigma^{(i)}_{j_{i}}}\cdots \sigma^{(n-2)}_{j_{n-2}}\tau_{j_{n-1},j_n}
.
\end{array}
\]

\noindent Similarly, evaluating  the determinant of $M_{n-1}$ at the parametrization, we obtain that the determinant of $M_{n-1}$ and $M_n$ are, respectively, the product of $\sum_{j_1,\ldots,j_{n-2}}\sigma^{(1)}_{j_1}\cdots\sigma^{(n-2)}_{j_{n-2}}
$ with 
\[
F_{n-1}:= \mathrm{det}
\left(
\begin{array}{cc}
\tau_{1,1}+\tau_{1,2}& 
\displaystyle\sum_{j_1,\ldots,j_{n-2},j_n}X^{(n-1)}_{j_1\cdots j_{n-2} 1 j_n}\sigma^{(1)}_{j_1}\cdots\sigma^{(n-2)}_{j_{n-2}}\tau_{1,j_{n}}
\\
\tau_{2,1}+\tau_{2,2} & \displaystyle\sum_{j_1,\ldots j_{n-2},j_n}X^{(n-1)}_{j_1\cdots j_{n-2} 2 j_n}\sigma^{(1)}_{j_1}\cdots\sigma^{(n-2)}_{j_{n-2}}\tau_{2,j_{n}}
\end{array}
\right),
\]

\[
F_n:=\mathrm{det}\left(
\begin{array}{cc}
\tau_{1,1}+\tau_{2,1}& 
\displaystyle\sum_{j_1,\ldots,j_{n-1}}X^{(n)}_{j_1\cdots j_{n-1} 1}\sigma^{(1)}_{j_1}\cdots\sigma^{(n-2)}_{j_{n-2}}\tau_{j_{n-1},1}
\\

\tau_{1,2}+\tau_{2,2} & \displaystyle\sum_{j_1,\ldots,j_{n-1}}X^{(n)}_{j_1\cdots j_{n-1} 2}\sigma^{(1)}_{j_1}\cdots\sigma^{(n-2)}_{j_{n-2}}\tau_{j_{n-1},2}

\end{array}
\right).
\]

\noindent  Given $n$ payoff tables, the Spohn CI variety can be written as $C_X~=~\mathbb{V}(F_1,\ldots,F_n)$. This variety is the object of study of this paper. In particular, $C_X$ is the intersection of $n$ divisors in $\Mc$ with multi-degree
$$(0,1,\ldots,1),\ldots, (1,\ldots,1,\underset{(i)}{0},1,\ldots,1), \ldots,(1,\ldots,1,0,1,1),(1,\ldots,1,2), (1,\ldots,1,2).$$ 

\begin{example}
For $n=2$, $\Mc=\P^3$ and, hence, $\Vc=\V$. From \cite[Theorem 8]{BI} we deduce that for generic payoff tables $C_X$ is an elliptic curve of degree $4$ in $\P^3$.\end{example}
\vspace{0.2cm}
\begin{proposition}\label{prop:dim 1}
For a generic game $X$, $C_X$ is a curve. In this case, we call $C_X$ the \emph{Nash conditional independence (CI) curve}.
\end{proposition}
\begin{proof}
Seeing the payoff tables $X^{(1)},\ldots,X^{(n)}$ as variables in $n$ copies of $\P^{2^n-1}$, we consider the variety $C$ defined as the projective subvariety  $\mathbb{V}(F_1,\ldots,F_n)$  of $\left(\P^{2^n-1}\right)^n\times~\Mc$. Here  $F_1,\ldots,F_n$ are the $n$ polynomials defining $C_X$ introduced above. Also, we consider the projection
\[\begin{array}{cccc}
\pi:&C&\longrightarrow &\left(\P^{2^n-1}\right)^n\\
 & (X^{(1)},\ldots,X^{(n)},p)&\longmapsto&(X^{(1)},\ldots,X^{(n)}).
\end{array}
\]
Note that, given a point $X\in \left(\P^{2^n-1}\right)^n$, the fiber of $X $ through $\pi$ is $C_X$.
Using that the dimension of the fibers of $\pi$ is upper semicontinuous, we obtain that there exists an open subset $U\subseteq\left(\P^{2^n-1}\right)^n $ such that, for every $X\in U$, $\dim C_X\leq 1$. Moreover, since for every $X$, $C_X$ is the zero locus of $n$ equations, $C_X$ has codimension at most $n$ in $\Mc$. Hence $\dim C_X \geq 1$ for every $X\in \left(\P^{2^n-1}\right)^n$. Thus, we conclude that for every $X\in U$, $\dim(C_X)=1$. In order to check that $U$ is dense, we prove that it is non-empty. 
For $k\leq n-2$ we fix $X^{(k)}$ such that the following equation holds:
\[
F_k = \displaystyle\prod_{i\neq k}^{n-2} (k\sigma^{(i)}_1-\sigma^{(i)}_2)(\tau_{1,1}+k\tau_{1,2}+k^2\tau_{2,1}+k^3\tau_{2,2}).
\]
Note that this is possible because if we expand the right-hand side of the above equation, we get a polynomial of the same format as $F_k$. Then, one can check that $\mathbb{V}(F_1,\ldots,F_{n-2})$ has dimension $3$ and its irreducible components are isomorphic to $\P^3$, $\P^1\times\P^2$, or $\P^1\times\P^1\times\P^1$. Following the same idea, one can fix $X^{(n-1)}$ and $X^{(n)}$ such that $F_{n-1}$ and $F_n$ are respectively as follows
\[
\begin{split}
&\prod_{i=1}^{n-2} ((n-1)\sigma^{(i)}_1-\sigma^{(i)}_2)(A_{n-1}\tau_{1,1}\tau_{2,1}+B_{n-1}\tau_{1,2}\tau_{2,1}+C_{n-1}\tau_{1,1}\tau_{2,2}+ D_{n-1}\tau_{1,2}\tau_{2,2}),
\\
&\prod_{i=1}^{n-2} (n\sigma^{(i)}_1-\sigma^{(i)}_2)
(A_{n}\tau_{1,1}\tau_{1,2}+B_{n}\tau_{2,2}\tau_{1,1}+C_{n}\tau_{1,2}\tau_{2,1}+ D_n\tau_{2,1}\tau_{2,2}),
\end{split}
\]
where $D_{n-1} = B_{n-1}+C_{n-1}-A_{n-1}$, $D_n=B_{n}+C_{n}-A_{n}$, and the coefficients $A_k,B_k,C_k$ depend linearly on $X^{(k)}$ for $k\geq n-1$. Finally, one can check that for generic coefficients $A_n,B_n,C_n,A_{n-1},B_{n-1},C_{n-1}$, the intersection of the threefold $\mathbb{V}(F_1,\ldots,F_{n-2})$ with $\mathbb{V}(F_{n-1},F_n)$ has dimension $1$. 
Hence, we  conclude that $U$ is non-empty.
\end{proof}

\begin{corollary}\label{co:gen com int}
For generic payoff tables, the Nash CI curve is a complete intersection of $n$ divisors of $\Mc$.
\end{corollary}

\para 

One may show that the specific Nash CI curve constructed in the proof of Proposition~\ref{prop:dim 1} is connected. By the semicontinuity theorem, we then obtain that for generic games, the Nash CI curve is connected. In the next section, we yet use a different argument to prove in Lemma~\ref{lem: connected} that the Nash CI curve is connected.

\para 

\begin{example}\label{ex: generic 3-player game}
We consider the 3-player game from \cite[Section 6.2]{CBMS} with the following payoff tensors
\begin{equation}\label{eq: 3-player generic}
\begin{matrix} 
   X^{(1)} = \\ X^{(2)} =  \\ X^{(3)}= 
\end{matrix}\,
\bordermatrix{     
            & 111        & 121      &   112   & 122     & 211     & 221     & 212     & 222 \cr
            & 0 & 6  & 11 & 1 & 6 & 4 & 6 &8    \cr
            & 12 & 7  & 6 & 8 & 10 & 12 & 8 &1    \cr
            & 11 & 11  & 3 & 3 & 0 & 14 & 2 &7    \cr
}.
\end{equation}    

A {\tt Macaulay2} computation shows that the Spohn CI variety $C_X$ is an irreducible curve of genus $3$ and degree $8$. Hence, $C_X$ is a Nash CI curve. Its equations in $\mathcal{M}_\mathcal{C}= \P^1\times\P^3$ are 
\[\begin{split}
 F_1 &= 6  \tau_{11}-5  \tau_{12}-2  \tau_{21}+7  \tau_{22},\\ 
 F_2 &= 
(5 \sigma^{(1)}_1\!\! -\!2 \sigma^{(1)}_2)  \tau_{11}  \tau_{21}\!-\!(\sigma^{(1)}_1 \!\!+\!4 \sigma^{(1)}_2 ) \tau_{12}  \tau_{21}\!+\!(4 \sigma^{(1)}_1\!\!+\!9 \sigma^{(1)}_2 ) \tau_{11}  \tau_{22}\!+\!(7 \sigma^{(1)}_2\!\!-\!2 \sigma^{(1)}_1  )\tau_{12}  \tau_{22},\\
\!\!F_3 &= (8 \sigma^{(1)}_1\!\! -\!2 \sigma^{(1)}_2 ) \tau_{11}  \tau_{12}\!+\!(8 \sigma^{(1)}_1 \!\!+\!12\sigma^{(1)}_2  )\tau_{12}  \tau_{21}\!+\!(8 \sigma^{(1)}_1\!\! -\!7 \sigma^{(1)}_2 ) \tau_{11}  \tau_{22}\!+\!(8 \sigma^{(1)}_1\!\!  +\!7 \sigma^{(1)}_2 ) \tau_{21}  \tau_{22}.
\end{split}
\]
The Spohn CI curve $C_X$ is contained in the four-dimensional irreducible Spohn variety of degree 4 whose ideal is generated by the determinants of three $2\times 2$ matrices $M_i$ depicted in (\ref{eq: Spohn matrices}). Moreover, $C_X$ contains the two totally mixed Nash equilibria of $X$ computed in \cite[Section 6.2]{CBMS}.
\end{example}

\section{Degree and genus of the Nash CI curve} \label{Section 3}

This section is devoted to the computation of the degree and the genus of the Nash CI curve~$C_X$. In particular, we prove that the Nash CI curve $C_X$ is connected. 
\para

\begin{lemma}\label{lem: degree of almost Nash curve}
For generic payoff tables, the degree of  $C_X$ is the coefficient of the monomial $x_1\cdots x_{n-2}x_{n-1}^3$ in the polynomial \begin{equation}\label{eq:poly deg}
    \displaystyle\prod_{i=1}^{n-2}\left( 
\sum_{k\neq i}^{n-1}x_k
\right)\left(2x_{n-1}+\sum_{k=1}^{n-2}x_k\right)^2\left(
\sum_{k=1}^{n-1}x_k\right).
\end{equation}
\end{lemma}

\begin{proof}
By Corollary~\ref{co:gen com int},  $C_X$ is the complete intersection of $n$ divisors $D_1,\ldots,D_n$ of $\Mc$, where $D_i:=\mathbb{V}(F_i)$. We compute the degree by multiplying the classes of these divisors with the class of $H\cap\Mc$ in the Chow ring of $\Mc$ where $H$ is a generic hyperplane of $\P^{2^n-1}$. 
Using K\"unneth's formula (see \cite[Theorem 2.10]{EH}), we obtain the Chow rings of $\Mc$ as follows: 
\[
A_{\bullet}(\Mc)\simeq\left(\displaystyle\bigotimes_{i=1}^{n-2}A_{\bullet}(\P^1)
\right)\otimes A_{\bullet}(\P^3)\simeq \Z[x_1,\ldots,x_{n-1}]/\langle x_1^2,\ldots,x_{n-2}^2,x_{n-1}^4\rangle.
\]
\noindent The classes $\left[D_i\right]$ in $A_{\bullet}(\Mc)$ correspond to the first Chern classes of the line bundles $\mathcal{O}_{\Mc}(D_i)$. 
Let $F$ be a multi-homogeneous polynomial in $\C[\sigma_{1}^{(1)},\ldots,\tau_{2,2}]$, and let $D$ be the divisor $D=\mathbb{V}(F)$ in $\Mc$. For $i\leq n-2$, we denote the degree of $F$ with respect to the variables $\sigma^{(i)}_1,\sigma^{(i)}_2$ by $d_i$ (resp. for $d_{n-1}$ and $\tau_{j_1j_2}$). Then, the line bundle associated to $D$ is 
\[
\mathcal{O}_{\Mc}(D)=\pi_1^*(\mathcal{O}_{\P^1}(d_1))\otimes\cdots\otimes\pi_{n-2}^*(\mathcal{O}_{\P^1}(d_{n-2}))\otimes \pi_{n-1}^*(\mathcal{O}_{\P^3}(d_{n-1})),
\]
where $\pi_i$ is the projection from $\Mc$ to the corresponding factor of the product.
We denote this line bundle by $\mathcal{O}_{\Mc}(d_1,\ldots,d_{n-1})$. In particular, we obtain that 
\[
\mathcal{O}_{\Mc}(D_i)=\left\{
\begin{array}{cl}
\mathcal{O}_{\Mc}(1,\ldots,1,\underset{(i)}{0},1,\ldots,1)& \text{ for } i\leq n-2,\\
& \\
\mathcal{O}_{\Mc}(1,\ldots,1,2)& \text{ for } i=n-1,n.
\end{array}
\right.
\]
Furthermore, since the first Chern class of $\mathcal{O}_{\Mc}(d_1,\ldots,d_{n-1})$ is
$d_1x_1+\cdots+d_{n-1}x_{n-1}$, we obtain the following:
\[
\left[D_i\right]=
\mathcal{O}_{\Mc}(D_i)=\left\{
\begin{array}{cl}
\displaystyle\sum_{k\neq i}^{n-1} x_k& \text{ for } i\leq n-2,\\

\displaystyle2x_{n-1}+\sum_{k=1}^{n-2} x_k& \text{ for } i=n-1,n.
\end{array}
\right.
\]
See \cite[Chapters 1 and 2]{EH} for more details on these computations.
Therefore, we deduce that 
\[
\left[D_1\right]\cdots\left[D_n\right]=\displaystyle\prod_{i=1}^{n-2}x_k\left( 
\sum_{k\neq i}^{n-1}x_k
\right)\left(2x_{n-1}+\sum_{k=1}^{n-2}\right)^2.
\]

Finally, using that for a generic hyperplane  $H$ it holds that $\left[H\cap \Mc \right]=x_1+\cdots+x_{n-1}$, we get that $\left[D_1\right]\cdots\left[D_n\right]\left[H\cap \Mc \right]$ corresponds to the class of the polynomial \eqref{eq:poly deg}
in $A_{\bullet}(\Mc)$ and hence the statement follows. 
\end{proof}

Now, we deal with the computation of the arithmetic genus of $C_X$. For this purpose, we aim to obtain the Euler characteristic of $\mathcal{O}_{C_X}$. By Corollary~\ref{co:gen com int}, we get the following exact sequence (Koszul complex):

\begin{equation}\label{eq:long exact seq}
\begin{array}{c}
\displaystyle 0  \rightarrow \mathcal{O}_{\Mc}\left(-\overset{n}{\underset{k=1}{\sum}} D_{i}\right) \rightarrow\!\! \bigoplus_{i_1<\cdots<i_{n-1}}^n\mathcal{O}_{\Mc}\left( -\overset{n-1}{\underset{k=1}{\sum}} D_{i_k} \right)\rightarrow\cdots\rightarrow

\\ 
\displaystyle \rightarrow \bigoplus_{i_1<i_2}^n\mathcal{O}_{\Mc}\left(-D_{i_1}-D_{i_2}\right)

 \rightarrow 
\bigoplus_{i}^n\mathcal{O}_{\Mc}\left(
-D_{i}
\right) \rightarrow 
\O_{\Mc} \rightarrow
\mathcal{O}_{C_X}\rightarrow 0.
\end{array}
\end{equation}

\para 

\noindent We define $\chi_k$ as the Euler characteristic of the $(k+2)$th term from the right of the previous exact sequence, i.e.\
\begin{equation}\label{eq:chi_k}
\!\!\chi_k:= \chi\left(
\bigoplus_{i_1<\cdots<i_{k}}^n\mathcal{O}_{\Mc}\left(-
D_{i_1}-\cdots -D_{i_k}
\right)
\right)=\!\!\sum_{i_1<\cdots< i_k}^n\!\! \chi\left(\mathcal{O}_{\Mc}\left(-
D_{i_1}-\cdots -D_{i_k}
\right)
\right).
\end{equation}
Then, using  \eqref{eq:long exact seq} we obtain that 
\begin{equation}\label{eq: alt sum euler}
\chi(\O_{C_x})=\chi(\O_{\Mc})+\displaystyle\sum_{k=1}^n (-1)^k\chi_k.
\end{equation}

\noindent By K\"unneth formula, and since $\chi(\O_{\P^n})=1$, we deduce that  $\chi(\O_{\Mc})=1$. Thus, it only remains to compute $\chi_k$. For the following lemma, our convention is that $\binom{a}{b} = 0$, if $a<b$.

\para 

\begin{lemma}\label{lemma:chi k}
For $3\leq k\leq n$, 
\[
\chi_k = \binom{n-2}{k}\chi(\O_{k,1})+ 2\binom{n-2}{k-1}\chi(\O_{k,2})+ \binom{n-2}{k-2}\chi(\O_{k,3}),
\]
where 
\begin{itemize}
    \item $\chi(\O_{k,1})=(-1)^{n+1}\binom{k-1}{3}(k-2)^k(k)^{n-2-k}$.
    \item $\chi(\O_{k,2})=(-1)^{n+1}\binom{k}{3}(k-2)^k(k-1)^{n-1-k}$.
    \item $\chi(\O_{k,3})=(-1)^{n+1}\binom{k+1}{3}(k-2)^k(k-2)^{n-k}$.
\end{itemize}
\noindent Moreover, we have that $\chi_1 = 0$ and $\chi_2 = (-1)^{n+1}$.
\end{lemma}
\begin{proof}
The idea  is to use K\"unneth formula as above. First of all, we observe that only the following three types of line bundles can appear in the sum of equation \eqref{eq:chi_k}:
\begin{itemize}
    \item If $i_1,\ldots,i_k\leq n-2$, the line bundle of the corresponding term of \eqref{eq:chi_k} is of the form 
    \[
    \O_{\Mc}(-k,\ldots,\underset{(i_1)}{-k+1},\ldots,\underset{(i_k)}{-k+1},\ldots,-k).
    \]
    Moreover, in \eqref{eq:chi_k} these terms appear $\binom{n-2}{k}$ times.
    \item If $i_1,\ldots,i_{k-1}\leq n-2$ and $i_k>n-2$, the line bundle of the corresponding term of \eqref{eq:chi_k} is of the form 
    \[
    \O_{\Mc}(-k,\ldots,\underset{(i_1)}{-k+1},\ldots,\underset{(i_{k-1})}{-k+1},\ldots,-k-1).
    \]
    Moreover, in \eqref{eq:chi_k} these terms appear $2\binom{n-2}{k-1}$ times.
    \item If $i_1,\ldots,i_{k-2}\leq n-2$ and $i_{k-1}>n-2$, the line bundle of the corresponding term of \eqref{eq:chi_k} is of the form 
    \[
    \O_{\Mc}(-k,\ldots,\underset{(i_1)}{-k+1},\ldots,\underset{(i_{k-2})}{-k+1},\ldots,-k-2).
    \]
    Moreover, in  \eqref{eq:chi_k} these terms appear $\binom{n-2}{k-2}$ times.
\end{itemize}

To deal with each of these three cases, we again use K\"unneth formula. In our setting, this formula states that for $d_1,\ldots,d_{n-1}\in\Z$, 
\[
h^k\left(\Mc,\O_{\Mc}(d_1,\ldots,d_{n-1})\right) = \displaystyle\sum_{k_1+\cdots+k_{n-1}=k}h^{k_{n-1}}\left(\P^3,\O_{\P^3}(d_{n-1})\right)\prod_{i=1}^{n-2}h^{k_i}\left(\P^1,\O_{\P^1}(d_i)\right).
\]
In particular, using  that all the first $n-2$ factors of $\Mc$ are isomorphic, we get that the Euler characteristic of each of the line bundles of the above cases does not depend on the choice of the $i_l$ that is lower or equal than $n-2$. Denoting the line bundles of the three cases by $\O_{k,1}$, $\O_{k,2}$, and $\O_{k,3}$ respectively, we get that 
\begin{equation}\label{eq:sum chi k}
\chi_k = \binom{n-2}{k}\chi(\O_{k,1})+ 2\binom{n-2}{k-1}\chi(\O_{k,2})+ \binom{n-2}{k-2}\chi(\O_{k,3}).
\end{equation}

To compute each of the terms of the above expression  we use that for $d>0$,
\[
h^k\left(\P^n,\O_{\P^n}(-d)\right)=\left\{
\begin{array}{cl}
0 & \text{ if } k\neq n \text{ or } d\leq n,\\
\dim\left(\frac{1}{x_0\cdots x_n}\C[x_0^{-1},\cdots,x_n^{-1}] \right)_d & \text{ if } k=n \text{ and } d> n.
\end{array}
\right.
\]
Note that $\dim\left(\frac{1}{x_0\cdots x_n}\C[x_0^{-1},\cdots,x_n^{-1}] \right)_d= \binom{d-1}{n}$. In particular, for $d_1,\ldots,d_{n-1}>0$ such that either $d_i\leq 1$ for some $i\leq n-2$ or  $d_i\leq 3$ for $i \geq n-1$, we obtain that 
$$h^k\left(\Mc,\O_{\Mc}(-d_1,\ldots,-d_{n-1})\right)=0\,\,\, \forall k.$$ On the contrary, if $d_i>1$ for every $i\leq n-2$ and $d_{n-1},d_n> 3$, we get that 
\begin{equation}\label{eq:coho lin mc}
h^k\left(\Mc,\O_{\Mc}(-d_1,\ldots,-d_{n-1})\right)=\left\{
\begin{array}{cl}
0 & \text{ if } k\neq n+1,  \\
\displaystyle
\binom{d_{n-1}-1}{3}\prod_{i=1}^{n-2}(d_i-1)& \text{ if } k=n+1. 
\end{array}
\right.
\end{equation}
Now, note that for all the line bundles appearing in $\chi_1$ it holds that  $d_{n-1}\leq 3$. Thus, we can conclude that $\chi_1 = 0$. Analogously, the only line bundle appearing in the expression of $\chi_2$ with a non-zero cohomology group is $\O(-D_{n-1}-D_n)=\O(-2,\ldots,-2,-4)$. Using equation \eqref{eq:coho lin mc} one deduces that the only non-zero cohomology group of this line bundle is the $(n+1)$th cohomology group and its dimension is $1$. Hence, we conclude that $\chi_2=(-1)^{n+1}$. A similar argument shows that, for $k\geq 3$, the only non-zero cohomology group of $\O_{k,1}$, $\O_{k,2}$, and $\O_{k,3}$ is the $(n+1)$th cohomology group and we deduce the following formulas:
\begin{itemize}
    \item $\chi(\O_{k,1})=(-1)^{n+1}\binom{k-1}{3}(k-2)^k k^{n-2-k}$.
    \item $\chi(\O_{k,2})=(-1)^{n+1}\binom{k}{3}(k-2)^k(k-1)^{n-1-k}$.
    \item $\chi(\O_{k,3})=(-1)^{n+1}\binom{k+1}{3}(k-2)^k(k-2)^{n-k}$.
\end{itemize}
The proof of the lemma follows from these expressions and equation \eqref{eq:sum chi k}.
\end{proof}

\noindent Finally, using this lemma and equation \eqref{eq: alt sum euler}, the following corollary can be derived:

\para 

\begin{corollary}
    For generic payoff games, the Euler characteristic of  $C_X$ is 
    \[
    \chi(\O_{C_X})= 1+(-1)^{n+1}+
    \displaystyle\sum_{k=3}^n (-1)^k\chi_k
    \]
    where $\chi_k$ is as in Lemma~\ref{lemma:chi k}.
\end{corollary}

\para 

Now that the Euler characteristic has been computed, we deal with the arithmetic genus. For this purpose, we first need to ensure that the Nash CI curve is connected.

\para 

\begin{lemma}\label{lem: connected}
For generic payoff tables, $h^0(C_X,\O_{C_X})=1$. In particular, $C_X$ is connected.
\end{lemma}
\begin{proof}
The idea is to split the exact sequence \eqref{eq:long exact seq} in short exact sequences and apply the long exact sequence of cohomology on each of them. To do so, let $\F_n$ be the first sheaf (starting from the left) appearing in the exact sequence \eqref{eq:long exact seq}. Similarly, let $\F_{n-1}$ be the second sheaf of the exact sequence, and so on. We denote the morphisms from $\F_i$ to $\F_{i-1}$ by $\phi_i$. Then, \eqref{eq:long exact seq} can be written as
\begin{equation}\label{eq:long exact seq 2}
\displaystyle
0\longrightarrow \F_n\overset{\phi_n}{\longrightarrow} \F_{n-1}\overset{\phi_{n-1}}{\longrightarrow}\cdots\overset{\phi_3}{\longrightarrow}
\F_{2}
\overset{\phi_2}{\longrightarrow}
\F_1
\overset{\phi_1}{\longrightarrow}
\O_{\Mc}\overset{\phi_0}{\longrightarrow}
\mathcal{O}_{C_X}\longrightarrow 0.
\end{equation}
Let $K_i$ be  the kernel of $\phi_i$. Then, the above exact sequence splits in the following $n$ short exact sequences:
\[
\begin{array}{ll}
 (E_0):    &  0\longrightarrow K_0\longrightarrow \O_{\Mc}\longrightarrow\O_{C_X}\longrightarrow 0,\\
 \noalign{\vspace*{1mm}}(E_1):    &  0\longrightarrow K_1\longrightarrow \F_{1}\longrightarrow K_0\longrightarrow 0,\\
 \,\,\,\,\vdots & \hspace*{2.7cm}\vdots \\
 (E_i):    &  0\longrightarrow K_i\longrightarrow \F_{i}\longrightarrow K_{i-1}\longrightarrow 0,\\
  \,\,\,\,\vdots & \hspace*{2.7cm}\vdots \\
 (E_{n-1}):    &  0\longrightarrow K_{n-1}\longrightarrow \F_{n-1}\longrightarrow K_{n-2}\longrightarrow 0.\\
\end{array}
\]
Now, we consider the long exact sequence in the cohomology of  $(E_0)$:
\[
0\longrightarrow H^0(\Mc,K_0)\longrightarrow H^0(\Mc,\O_{\Mc})\longrightarrow H^0(C_X,\O_{C_X})\longrightarrow H^1(\Mc,K_0)\longrightarrow\cdots.
\]
By the exactness of this sequence, if $h^1(\Mc,K_0)=0$, we would get a surjection 
\[
H^0(\Mc,\O_{\Mc})\longrightarrow H^0(C_X,K_0)\longrightarrow 0.
\]
Since  $h^0(\Mc,\O_{\Mc})=1$, this would imply that $h^0(C_X,\O_{C_X})=1$. Thus, it is enough to check that $h^1(\Mc,K_0)=0$. To do so, we focus on the long exact sequence in cohomology arising from $(E_1)$:
\[
\cdots \longrightarrow H^1(\Mc,\F_1)\longrightarrow H^1(\Mc,K_0)\longrightarrow H^2(\Mc,K_1)\longrightarrow H^2(\Mc,\F_1)\longrightarrow\cdots .
\]
By the computations performed in the proof of Lemma~\ref{lemma:chi k}, we know that for $j\geq 1$ and  $k\leq n$, $h^k(\Mc,\F_j)=0$. Thus, from the above sequence, we conclude that $H^1(\Mc,K_0)\simeq H^2(\Mc,K_1)$. Recursively, we get that $H^i(\Mc,K_{i-1})\simeq H^{i+1}(\Mc,K_i)$ for every $i\leq n-1$.  In particular, by the exactness of \eqref{eq:long exact seq 2},  we have that $K_{n-1} = \F_n$. Hence, we obtain that
\[
H^1(\Mc,K_0)\simeq H^2(\Mc,K_1)\simeq \cdots\simeq H^n(\Mc,K_{n-1})=H^n(\Mc,\F_n).
\]
Now, the proof follows from the vanishing of the cohomology group $H^n(\Mc,\F_n)$.
\end{proof}

\para 

Since we have shown $C_X$ is connected, we may derive the arithmetic genus of the Nash CI curve.

\para

\begin{corollary}\label{cor: genus of the generic curve}
    For generic payoff tables, $C_X$ is a connected curve of arithmetic genus
    \[p_a(C_X)=
    (-1)^{n}+
    \displaystyle\sum_{k=3}^n (-1)^{k+1}\chi_k
    \]
    where $\chi_k$ is as in Lemma~\ref{lemma:chi k}.
\end{corollary}

\para

\begin{example}
Consider the 3-player game from Example~\ref{ex: generic 3-player game}. By Lemma~\ref{lem: degree of almost Nash curve}, we obtain the degree of $C_X$ is $8$ which is the coefficient of the monomial $x_1 x_2^3$ in the polynomial
$$x_2 \left( 2 x_2 + x_1 \right)^2 (x_1 + x_2).$$
By Corollary \ref{cor: genus of the generic curve} and Lemma \ref{lemma:chi k}, the genus of $C_X$ is 
\[
p_a(C_X) = (-1)^3+\chi_3=-1+4 =3.
\]
Note that this computation matches with the {\tt Macaulay2} computation in Example \ref{eq: 3-player generic}.
\end{example}

\para 

Using Corollary \ref{cor: genus of the generic curve} in Table~\ref{table-2}, we can determine effectively the genus of the Nash CI curve for different values of $n$.

\para 

\begin{table}[h]
\centering
\begin{tabular}{|c|c|c|}
 \hline
 $n$ & genus & degree\\
 \hline
 3  &  3  & 8\\
4  &  23 & 30\\
 5 & 175 & 146\\
 6 & 1469 & 880\\
7  & 13491 & 6276\\
 8 & 135859 & 51562\\
 9 & 1494879 & 478670\\
 \hline
 \end{tabular}
 \caption{Genus and degree of the Nash CI curve $C_X$.}
              \label{table-2}
 \end{table}

\para

\begin{remark}
 The genus of $C_X$ can be also computed combinatorially via \cite[Theorem 1]{khovanskii}, and more generally by the motivic arithmetic genus formula in \cite{RHN}. The main idea is to determine the discrete mixed volume of Newton polytopes $\rm Newt(F_i)$ for $i\in[n]$.

Another method for computing the genus of the Nash CI curve is using the adjunction formula. Let $i:C_X\hookrightarrow \Mc$ be the closed immersion of $C_X$ in $\Mc$. Applying the adjunction formula (see \cite[Section 1.4.3]{EH}) we get that 
\begin{equation}\label{eq: adj form}
\omega_{C_X}\!\!\simeq i^{*}\!\left(\omega_{\Mc}\!\otimes\!\O_{\Mc}(D_1)\!\otimes\!\cdots\!\otimes\!\omega_{\Mc}(D_n)\right)\simeq i^{*}\!\left(\O_{\Mc}(n-3,\ldots,n-3,n-2) \right).
\end{equation}
 Since $C_X$ is of complete intersection, it is Gorenstein and hence, \[
\mathrm{deg}(\omega_{C_X}) = 2p_a(C_X)-2.
\]
As a consequence of equation \eqref{eq: adj form}, one can compute the degree of $\omega_{C_X}$ using the Chow ring of $\Mc$ as in Lemma~\ref{lem: degree of almost Nash curve}. Using this method one obtain that $2p_a(C_X)-2$ is equal to the coefficient of the monomial $x_1\cdots x_{n-2}x_{n-1}^3$ in the polynomial 
\[
\displaystyle\prod_{i=1}^{n-2}\left( 
\sum_{k\neq i}^{n-1}x_k
\right)\left(2x_{n-1}+\sum_{k=1}^{n-2}x_k\right)^2\left((n-3)
\sum_{k=1}^{n-2}x_k+(n-2)x_{n-1}\right).
\]
However, our computations in {\tt Macaulay2} show that this method is less effective than the formula in Corollary~\ref{cor: genus of the generic curve}.
\end{remark}

\section{Universality}\label{Sec 4}

This section is devoted to studying the universality of  Spohn CI varieties for one-edge undirected graphical models. 
So far, we have been working with generic payoff tables and we have shown that, in this setting, $C_X$ is a curve. Now, we abandon this assumption so that $C_X$ can have an arbitrary positive dimension. We explore two different notions of universality.

\para 
The intersection of $\V$ and  the Segre variety $\left(\P^1\right)^n$ in the probability simplex is the set of totally mixed Nash equilibria. On the other hand, $\V \cap \left(\P^1\right)^n$ is the intersection of $n$ divisors of the Segre variety with multi-degrees $(0,1,\ldots,1),\ldots,(1,\ldots,1,0)$. We consider the map that sends an $n$--player game to these $n$ divisors. The universality of divisors states that this map is surjective (\cite[Corollary 6.7]{CBMS}). In other words, any $n$ divisors with the above multi-degree can be obtained from a game. In Section~\ref{sec: universality of divisors}, we explore this notion of universality for Spohn CI varieties for one-edge undirected graphical models. In this case, one does not intersect the Spohn variety with $\left(\P^1\right)^n$, but with the Segre variety $\Mc$. We prove that this notion of the universality of divisors does not hold (Proposition~\ref{cor: universality of divisors}). Moreover, we characterize the divisors that arise from the $n$ divisors defining $C_X$. This allows us to prove in Theorem~\ref{thm: irreducible Nash CI} that the Nash CI curve $C_X$ is smooth and irreducible. 
\para

In Section~\ref{sec: affine universality}, we study the affine universality of the Spohn CI varieties for one-edge undirected graphical models, as Datta does in \cite{datta} for the set of totally mixed Nash equilibria. In \cite{datta}, the author shows that, given a real affine algebraic variety $S$, there exists a game with binary choices such that its set of totally mixed Nash equilibria has an affine open subset isomorphic to $S$. Here, we derive an analogous result for $C_X$. Since the dimension of $C_X$ is at least $1$, the universality theorem cannot hold for $0$-dimensional real algebraic varieties. However, we offer two different approaches for dealing with this challenge. First, in Corollary~\ref{co:affine univ 1} we show that for a real affine algebraic variety $S$, there exists a  game with binary choices and an affine open subset $W_X$ of $C_X$ such that $S~\times~\mathbb{R}^1\simeq W_X$. On the other hand, we prove that in Theorem~\ref{thm: universality of Spohn CI} for $C_X$ the universality theorem holds for real affine algebraic varieties $S\subset \mathbb{R}^m$ defined by $k$ polynomials with $k<m$. 
 
\subsection{Universality of divisors and smoothness of Nash CI curve}\label{sec: universality of divisors}
Let $X^{(1)},\ldots, X^{(n)}\in\mathbb{R}^{2}\times\cdots\times\mathbb R^2$ be the payoff tables of an $n$-player game with binary choices.  Following \cite{datta}, its set of totally mixed Nash equilibria is the intersection of the probability simplex with the variety denoted by $N_{X}\subseteq \left(\P^1\right)^n\subset \P^{2^n-1}$ which is defined by the following $n$ equations 
\begin{equation}\label{eq:nash poly}
G_i = \displaystyle\sum_{j_1,\ldots,\widehat{j_i},\ldots,j_n}
\left(X^{(i)}_{j_1\cdots2\cdots j_n}- X^{(i)}_{j_1\cdots1\cdots j_n} \right)
\tilde{\sigma}^{(1)}_{j_1}\cdots \widehat{\tilde{\sigma}^{(i)}_{j_{i}}}\cdots \tilde{\sigma}^{(n)}_{j_{n}}, \,\,\text{ for } i\in[n],
\end{equation}
where $[\tilde{\sigma}^{(i)}_1:\tilde{\sigma}^{(i)}_2]$ are the coordinates of the $i$th factor of $\left(\P^1\right)^n$.
We denote the divisor of $\left(\P^1\right)^n$  corresponding to $G_i$ by $\tilde{D}_i$. Then, $N_{X}=\tilde{D}_1\cap\cdots\cap \tilde{D}_n$. For $1\leq i\leq n$, let $\tilde{V}_i$ be the vector space of multi-homogeneous polynomials with multi-degree $(1,\ldots,1,\underset{(i)}{0},1,\ldots,1)$, i.e.\
\[
\tilde{V}_i := H^0\left(\left(\P^1\right)^n,\O_{\left(\P^1\right)^n}(1,\ldots,0,\ldots,1)\right).
\]
Note that $\tilde{V}_i$ has dimension $2^{n-1}$.
 Let $D_i'$ be a divisor of $\left(\P^1\right)^n$ defined by a polynomial in $G_i'\in \tilde{V}_i$.

 Then, there exists a payoff table $X^{(i)}$ such that $G_i=G_i'$, and hence $\tilde{D}_i=D_i'$ (see \cite[Corollary 6.7]{CBMS}). 
 In other words, we get that the linear map
 \[
 \begin{array}{ccc}
 \mathbb{R}^{2^n}&\longrightarrow&\tilde{V}_i\\
\tilde{ X}^{(i)}&\longmapsto&G_i
 \end{array}
 \]

 is surjective. Identifying the set of payoff tables with the vector space $\left(\mathbb{R}^{2^n}\right)^n$, the universality in the Nash setting is derived from the surjectivity of the linear map
 \begin{equation}\label{eq:surj 1}
 \begin{array}{ccc}
\left( \mathbb{R}^{2^n}\right)^n&\longrightarrow&\tilde{V}_1\times\cdots\times \tilde{V}_n\\
 (X^{(1)},\ldots,X^{(n)})&\longmapsto&(G_1,\ldots,G_n).
 \end{array}
 \end{equation}
Our goal is to study whether the analogous map for the variety $C_X$ is surjective or not. For this purpose,  for $i\leq n-2$, let $V_i$ be the vector space of multi-homogeneous polynomials of multi-degree $(1,\dots,\underset{(i)}{0},\ldots,1)$, i.e.\
\[
V_i := H^0\left(\Mc,\O_{\Mc}(1,\ldots,\underset{(i)}{0},\ldots,1)\right).
\]
Similarly, we define the vector spaces $V_{n-1}$ and $V_{n}$ as
\[
V_{n-1}=V_n= H^0\left(\Mc,\O_{\Mc}(1,\ldots,1,2)\right).
\]
Note that the dimension of $V_{n-1}$ and $V_{n}$ is $10 \cdot 2^{n-1}$. For every $i$, we consider the linear map 
\[
 \begin{array}{cccc}
 \phi_i:&\mathbb{R}^{2^n}&\longrightarrow&V_i\\
 &X^{(i)}&\longmapsto&F_i.
 \end{array}
\]

Then, the analogous map to \eqref{eq:surj 1} can be constructed as
 \begin{equation}\label{eq:surj 2}
 \begin{array}{ccc}
\left( \mathbb{R}^{2^n}\right)^n&\longrightarrow&V_1\times\cdots\times V_n\\
 (X^{(1)},\ldots,X^{(n)})&\longmapsto&(F_1,\ldots,F_n).
 \end{array}
 \end{equation}
 In Proposition~\ref{cor: universality of divisors} we compute the image of the map (\ref{eq:surj 2}), and we conclude that it is not surjective. \\
\para
 \begin{proposition}\label{cor: universality of divisors}
 Let $\Lambda_i$ be the linear system defined by the image of $\phi_i$. The image $\Lambda_1 \times \ldots \times \Lambda_n$ of the linear map \eqref{eq:surj 2} has dimension $2^{n-1}(n+1)$. In particular, the map \eqref{eq:surj 2} is not surjective.
 \end{proposition}
 \begin{proof}
     First of all, we note that, as in the Nash case, $\phi_i$ is surjective for $i\leq n-2$. Thus, for $i\leq n-2$, $\Lambda_i = V_i$ and $\dim \Lambda_i = 2^{n-1}$. For $i=n-1$, we can rewrite $V_{n-1}$ as the tensor product
     \[H^0(\P^3,\O_{\P^1}(2))\otimes 
    \left( H^0(\P^1,\O_{\P^1}(1))\right)^{\otimes (n-2)}.
     \]
      Expanding  $F_{n-1}$, we obtain 
     \begin{equation}\label{eq: Fn-1} \begin{array}{ccl}
     F_{n-1} & =& \displaystyle\sum_{j_1,\ldots,j_{n-2}}\sigma_{j_1}^{(1)}\cdots \sigma_{j_{n-2}}^{(n-2)}\left(
     A_{j_1\cdots j_{n-2}}^{(n-1)}\tau_{1,1}\tau_{2,1}+
     B_{j_1\cdots j_{n-2}}^{(n-1)}\tau_{1,2}\tau_{2,1}+\right. \\
      \noalign{\vspace*{1mm}} 
     & & \left. 
     C_{j_1\cdots j_{n-2}}^{(n-1)}\tau_{1,1}\tau_{2,2}+
     D_{j_1 \cdots j_{n-2}}^{(n-1)}\tau_{1,2}\tau_{2,2}
     \right)
     \end{array}
     \end{equation}
     where the coefficients are defined as
     \[
     \begin{array}{ll}
     A_{j_1\cdots j_{n-2}}^{(n-1)}= X_{j_1 \cdots j_{n-2}21}^{(n-1)}-X_{j_1 \cdots j_{n-2}11}^{(n-1)},& 
     C_{j_1\cdots j_{n-2}}^{(n-1)} = X_{j_1 \cdots j_{n-2}22}^{(n-1)}-X_{j_1 \cdots j_{n-2}11}^{(n-1)},
     \\ \noalign{\vspace*{1mm}} 
     B_{j_1\cdots j_{n-2}}^{(n-1)}= X_{j_1 \cdots j_{n-2}21}^{(n-1)}-X_{j_1 \cdots j_{n-2}12}^{(n-1)},&
     \!\!\!D_{j_1 \cdots j_{n-2}}^{(n-1)}\!\!=B_{j_1\cdots j_{n-2}}^{(n-1)}\!+C_{j_1\cdots j_{n-2}}^{(n-1)}\!-A_{j_1\cdots j_{n-2}}^{(n-1)}.
     \end{array}
     \]
     From this expression we deduce that $\Lambda_{n-1} = W_{n-1}\otimes\left( H^0(\P^1,\O_{\P^1}(1))\right)^{\otimes (n-2)}$, where 
     \[
     W_{n-1} = \{ A\tau_{1,1}\tau_{2,1}+
     B\tau_{1,2}\tau_{2,1}+
     C\tau_{1,1}\tau_{2,2}+
     (B+C-A)\tau_{1,2}\tau_{2,2}: \text{ for } A,B,C\in\mathbb{R}  \}. 
     \]
     Thus, $W_{n-1}$ is a linear subspace of dimension $3$ and hence,
     $
     \dim \Lambda_{n-1} = 3\cdot  2^{n-2}
     $.

     Analogously, for $F_n$ we get the expression 
     
     \begin{equation}\label{eq: Fn} \begin{array}{ccl}
     F_{n} & =& \displaystyle\sum_{j_1,\ldots,j_{n-2}}\sigma_{j_1}^{(1)}\cdots \sigma_{j_{n-2}}^{(n-2)}\left(
     A_{j_1\cdots j_{n-2}}^{(n)}\tau_{1,1}\tau_{1,2}+
     B_{j_1\cdots j_{n-2}}^{(n)}\tau_{1,2}\tau_{2,1}+\right. \\
      \noalign{\vspace*{1mm}} 
     & & \left. 
     C_{j_1,\ldots,j_{n}}^{(n)}\tau_{1,1}\tau_{2,2}+
     D_{j_1,\ldots,j_{n}}^{(n)}\tau_{2,1}\tau_{2,2}
     \right)
     \end{array}
     \end{equation}
     where the coefficients are defined as
     \[
     \begin{array}{ll}
     A_{j_1\cdots j_{n-2}}^{(n)}= X_{j_1 \cdots j_{n-2}12}^{(n)}-X_{j_1 \cdots j_{n-2}11}^{(n)},& 
     C_{j_1\cdots j_{n-2}}^{(n)} = X_{j_1 \cdots j_{n-2}22}^{(n)}-X_{j_1 \cdots j_{n-2}11}^{(n)},
     \\ \noalign{\vspace*{1mm}} 
     B_{j_1\cdots j_{n-2}}^{(n)}= X_{j_1 \cdots j_{n-2}12}^{(n)}-X_{j_1 \cdots j_{n-2}21}^{(n)},&
     \!\!\!D_{j_1 \cdots j_{n-2}}^{(n)}\!\!=B_{j_1\cdots j_{n-2}}^{(n)}\!+C_{j_1\cdots j_{n-2}}^{(n)}\!-A_{j_1\cdots j_{n-2}}^{(n)}.
     \end{array}
     \]
     The dimension of $\Lambda_n$ follows from this expression as before. Then,  the image of the map \eqref{eq:surj 2} is $\Lambda_1\times\cdots \times \Lambda_n$ and it has dimension $2^{n-1}(n+1)$. It follows that the map \eqref{eq:surj 2} is not surjective since the dimension of $V_1\times \cdots \times V_n$ is $2^{n-1}(n+8)$. \end{proof}

 As a consequence of this proposition, we conclude that, in contrast with the Nash case,  the universality for divisors does not hold. In particular, this means that the linear systems $\Lambda_1,\ldots,\Lambda_n$ might have base locus. From Bertini's theorem, we deduce that for generic payoff tables, $C_X$ is smooth away from these base loci. Hence, in order to study the smoothness of $C_X$ we need to compute the base locus of each of these linear systems. Since $\phi_i$ is surjective for $i\in[n-2]$, we get that $\Lambda_i$ is complete for $i\in[n-2]$ and, hence, base point free. The next lemma computes the base locus of $\Lambda_{n-1}$ and $\Lambda_n$.
 \para 
\begin{lemma}
 \begin{enumerate}
     \item[(1)]  The base locus of $\Lambda_{n-1}$ is $\left(\P^1\right)^{n-2}\times \left(L_1^{(n-1)}\cup L_2^{(n-1)}\cup L_3^{(n-1)}\right)$ where $L_1^{(n-1)}, L_2^{(n-1)}, L_3^{(n-1)}$ are the lines of $\P^3$ defined by the equations 
 \[
\begin{array}{cccc}
\{\tau_{1,1}=\tau_{1,2}= 0\}, & \{ \tau_{2,1}=\tau_{2,2}= 0\},&
\text{ and } & \{\tau_{1,1}+\tau_{1,2}=\tau_{2,1}+\tau_{2,2} = 0 \}
\end{array}
 \]
 respectively.
 \item[(2)] The base locus of $\Lambda_{n}$ is $\left(\P^1\right)^{n-2}\times \left(L_1^{(n)}\cup L_2^{(n)}\cup L_3^{(n)}\right)$ where $L_1^{(n)}, L_2^{(n)}, L_3^{(n)}$ are the lines of $\P^3$ defined by the equations 
 \[
\begin{array}{cccc}
\{\tau_{1,1}=\tau_{2,1}= 0\}, & \{ \tau_{1,2}=\tau_{2,2}= 0\},&
\text{ and } & \{\tau_{1,1}+\tau_{2,1}=\tau_{1,2}+\tau_{2,2} = 0 \}
\end{array}
 \]
 respectively.
 \end{enumerate}
 \end{lemma}
 \begin{proof}
     As in the proof of Proposition~\ref{cor: universality of divisors}, we write $\Lambda_{n-1}$ as the tensor product
      \[
     \Lambda_{n-1} =  W_{n-1}\otimes \left( H^0(\P^1,\O_{\P^1}(1))\right)^{\otimes (n-2)}.
     \]
     Thus, to compute the base locus of $\Lambda_{n-1}$ it is enough to compute the base locus of $W_{n-1}$ in $\P^3$. The elements of this linear system are polynomials of the form 
     \[
     A\tau_{1,1}\tau_{2,1} + B\tau_{1,2}\tau_{2,1} + C\tau_{1,1}\tau_{2,2} + (B+C-A)\tau_{1,2}\tau_{2,2}
     \]
     for $[A:B:C]\in\P^2$. In particular, for $[A:B:C]\in\{[1:0:0],[0:1:0],[0:0:1]\}$ we obtain the polynomials $\tau_{1,1}\tau_{2,1}-\tau_{1,2}\tau_{2,2},\tau_{1,2}(\tau_{2,1}+\tau_{2,2}),\tau_{2,2}(\tau_{1,1}+\tau_{1,2})$, which generate $W_i$. Then, one can check that the base locus of these three polynomials is exactly $L_1^{(n-1)}\cup L_2^{(n-1)}\cup L_3^{(n-1)}$. Thus, we conclude that the base locus of $\Lambda_{n-1}$ is $\left(\P^1\right)^{n-2}\times \left(L_1^{(n-1)}\cup L_2^{(n-1)}\cup L_3^{(n-1)}\right)$. A similar computation derives the statement for~$\Lambda_{n}$.
 \end{proof}
 
 \para 
 
 For generic payoff tables, $\mathbb{V}(F_1,\ldots,F_{n-2},F_{n-1})$ and $\mathbb{V}(F_1,\ldots,F_{n-2},F_{n})$ are surfaces and their intersection is the Nash CI curve $C_X$. By Bertini's theorem, we obtain that for generic payoff tables $\mathbb{V}(F_1,\ldots,F_{n-2},F_{n-1})$ (respectively $\mathbb{V}(F_1,\ldots,F_{n-2},F_{n})$) is smooth away from the base loci $\Lambda_{n-1}$ (respectively $\Lambda_{n}$). The next proposition states that these surfaces are indeed smooth.

\begin{proposition}\label{prop: smooth loci n-1}
For generic payoff tables, $\mathbb{V}(F_1,\ldots,F_{n-2},F_{n-1})$ and $\mathbb{V}(F_1,\ldots,F_{n-2},F_{n})$ are smooth surfaces.
\end{proposition}
 \begin{proof}
 Denote $\mathbb{V}(F_1,\ldots,F_{n-2},F_{n-1})$ and $\mathbb{V}(F_1,\ldots,F_{n-2},F_{n})$ by $S_{X,1}$ and $S_{X,2}$ respectively. We prove the result for $S_{X,2}$ (the analogous proof works for $S_{X,1}$). By Bertini's theorem, it is enough to check smoothness at the points of the base locus of $\Lambda_n$. Let $\Omega_i= S_{X,2}\cap\left(\left(\P^1\right)^{n-2}\times L_i^{(n)}\right)$. Then, $\Omega_1\cup \Omega_2\cup \Omega_3$ is the intersection of $S_{X,2}$ with the base locus of $\Lambda_n$. We check the Jacobian criterion at the points of $\Omega_1\cup \Omega_2\cup \Omega_3$.

We will focus on $\Omega_1$  (similar reasoning works  for $\Omega_2$ and $\Omega_3$). First, we prove that $\Omega_1$ is smooth.
Let $\tilde{F}_i$ be the polynomial resulting from restricting $F_i$ to $\Omega_1$, i.e., substituting $\tau_{1,1}$ and $\tau_{2,1}$ by $0$ in $F_i$. Then, $\Omega_1=\mathbb{V}(\tilde{F}_1,\ldots,\tilde{F}_{n-2})$.
Moreover, $\tilde{F}_1,\ldots,\tilde{F}_{n-2}$ are generic elements of complete linear systems of $\left(\P^1\right)^{n-2}\times L_1^{(n)}$. By Bertini's Theorem, $\Omega_1$ is smooth. 

\para

Now  we check the Jacobian criterion for $S_{X,2}$ at $\Omega_1$. 
Let $J_{S_{X,2}}(x)$ be the Jacobian matrix of $F_1,\ldots,F_{n-2},F_n$ at $x$, written as:

\[J_{S_{X,2}}(x):=
\begin{blockarray}{cccc|c}
    & F_1&\cdots&F_{n-2}&F_n \\
    \begin{block}{c(ccc|c)}
    \vdots &  \BAmulticolumn{3}{c|}{\multirow{5}{*}{$A(x)$}}&\BAmulticolumn{1}{c}{\multirow{5}{*}{$B(x)$}}\\
    \frac{\partial}{\partial \sigma_{j_i}^{(i)}}&  &&&\\
    \vdots &  &&&\\
    \frac{\partial}{\partial \tau_{1,2}} &   &&&\\[3pt]
    \frac{\partial}{\partial \tau_{2,2}}  &  &&&
    \\[6pt]
    \cline{1-5}
    \rule{0pt}{3ex}    
    \frac{\partial}{\partial \tau_{1,1}}  & \BAmulticolumn{3}{c|}{\multirow{2}{*}{$C(x)$}}&\BAmulticolumn{1}{c}{\multirow{2}{*}{$D(x)$}}\\[3pt]
    \frac{\partial}{\partial \tau_{2,1}}  &  &&&\\
    \end{block}
  \end{blockarray}
  \,\,\,
  .
\]

We observe that for $x\in \Omega_1$, $B(x)=0$ and $A(x)$ is the Jacobian of $\tilde{F}_1,\ldots,\tilde{F}_{n-2}$ w.r.t.\ $\sigma_{1}^{(1)},\sigma_{2}^{(1)},\ldots,\sigma_{1}^{(n-2)},\sigma_{2}^{(n-2)},\tau_{1,2},\tau_{2,2}$. Since, $\Omega_1$ is smooth, $A(x)$ has maximal rank. Hence, $J_{S_{X,2}}(x) $ has maximal rank if and only if $D(x)$ has maximal rank. Thus, a point $x\in\Omega_1$ is a singular point of $S_{X,2}$ if and only if it lies in  $\mathbb{V}(\frac{\partial F_n}{\partial \tau_{1,1}},\frac{\partial F_n}{\partial \tau_{2,1}})$. These two derivatives are 
\[\begin{array}{c}\displaystyle
    \frac{\partial F_n}{\partial \tau_{1,1}}=\displaystyle
  \sum_{j_1,\ldots,j_{n-2}}\sigma_{j_1}^{(1)}\cdots\sigma_{j_{n-2}}^{(n-2)}(A^{(n)}_{j_1\cdots j_{n-2}}\tau_{1,2}+ C^{(n)}_{j_1\cdots j_{n-2}}\tau_{2,2} ), \\ \displaystyle
    \frac{\partial F_n}{\partial \tau_{2,1}} = \displaystyle
  \sum_{j_1,\ldots,j_{n-2}}\sigma_{j_1}^{(1)}\cdots\sigma_{j_{n-2}}^{(n-2)}(B^{(n)}_{j_1\cdots j_{n-2}}\tau_{1,2}+ D^{(n)}_{j_1\cdots j_{n-2}}\tau_{2,2} ).
    \end{array}
    \]

As before, $ \frac{\partial F_n}{\partial \tau_{1,1}}$ is a generic element of a complete linear system of $\left(\P^1\right)^{n-2}\times L_1^{(n)}$.
In particular, we deduce that for generic payoff tables, $\mathbb{V}(\tilde{F}_1,\ldots,\tilde{F}_{n-2}, \frac{\partial F_n}{\partial \tau_{1,1}})$ has dimension $0$. We write $ \frac{\partial F_n}{\partial \tau_{2,1}}$ as 
     \[
      \frac{\partial F_n}{\partial \tau_{2,1}} = \displaystyle
  \sum_{j_1,\ldots,j_{n-2}}\sigma_{j_1}^{(1)}\cdots\sigma_{j_{n-2}}^{(n-2)}(B^{(n)}_{j_1\cdots j_{n-2}}(\tau_{1,2}+\tau_{2,2})+ (C^{(n)}_{j_1\cdots j_{n-2}}-A^{(n)}_{j_1\cdots j_{n-2}})\tau_{2,2} ).
     \]
     For generic coefficients $A^{(n)}_{j_1\cdots j_{n-2}}$ and $C^{(n)}_{j_1\cdots j_{n-2}}$, $\mathbb{V}(\tilde{F}_1,\ldots,\tilde{F}_{n-2}, \frac{\partial F_n}{\partial \tau_{1,1}},\tau_{1,2}+\tau_{2,2})$ is empty. Hence, for $x\in \mathbb{V}(\tilde{F}_1,\ldots,\tilde{F}_{n-2}, \frac{\partial F_n}{\partial \tau_{1,1}})$, $ \frac{\partial F_n}{\partial \tau_{2,1}}(x)=0$ defines a proper subvariety in the space of coefficients $B^{(n)}_{j_1\cdots j_{n-2}}$.
     Thus, for generic payoff tables,  $\mathbb{V}(\tilde{F}_1,\ldots,\tilde{F}_{n-2}, \frac{\partial F_n}{\partial \tau_{1,1}}, \frac{\partial F_n}{\partial \tau_{2,1}})$ is empty and we conclude that $S_{X,2}$ is smooth at $\Omega_1$. The same argument follows for $\Omega_2$ and $\Omega_3$.\end{proof}
 \para 

In particular, applying Bertini's theorem to $\mathbb{V}(F_1,\ldots,F_{n-2},F_{n-1})$ and the linear system $\Lambda_n$, we obtain that for generic payoff tables, the singular locus of $C_X$ lies in $\mathbb{V}(F_1,\ldots,F_{n-2},F_{n-1})$ with the base locus of $\Lambda_n$. Similarly, applying the same argument to $\mathbb{V}(F_1,\ldots,F_{n-2},F_{n})$ and $\Lambda_{n-1}$, we deduce that the singular locus lies in the intersection of $\mathbb{V}(F_1,\ldots,F_{n-2},F_{n})$ with the base locus of $\Lambda_{n-1}$. This implies that for generic payoff tables, the singular locus of $C_X$ lies in the intersection of the base loci of $\Lambda_{n-1}$ and $\Lambda_n$. This intersection is 
 \[
  \left(\P^1\right)^{n-2}\times \{q_1,q_2,q_3,q_4,q_5\},
 \]
where $q_1,\ldots,q_5$ are $[1:0:0:0],[0:1:0:0], [0:0:1:0],[0:0:0:1],[1:-1:-1:1]$ respectively. In the next theorem, we deduce the smoothness from studying locally the smoothness at the points of this intersection.

\para 

\begin{theorem}\label{thm: irreducible Nash CI}
For generic payoff tables, the Nash CI curve $C_X$ is smooth and irreducible.
\end{theorem}
\begin{proof}
By Lemma~\ref{lem: connected}, $C_X$ is connected. Therefore, it is enough to prove the smoothness of $C_X$ in order to conclude the irreducibility. For $i\in [5]$ let $S_i$ be the intersection of $ C_X$ with $ \left(\P^1\right)^{n-2}\times \{q_i\}$. By Bertini's theorem, the singular locus of $C_X$ lies in the intersection of $C_X$ with $ S_1\cup\cdots\cup S_5$. 
The strategy we follow is to apply locally the Jacobian Criterion on the points of this intersection. The reasoning is similar to the proof of Proposition~\ref{prop: smooth loci n-1}.
We analyze the smoothness at the points in $S_1$ (similarly for $S_2,\ldots, S_5$). First, one can check using Bertini's Theorem that $S_1$ is smooth.  For $x\in S_1$, the Jacobian matrix of $F_1\ldots,F_n$ with respect to $\sigma_{1}^{(1)},\sigma_{2}^{(1)},\ldots,\sigma_{1}^{(n-2)},\sigma_{2}^{(n-2)},\tau_{1,2},\tau_{2,1},\tau_{2,2} $ at $x$ is of the form

\[
  J_{C_X}(x):=\begin{blockarray}{cccc|cc}
    & F_1&\cdots&F_{n-2}&F_{n-1}&F_n \\
    \begin{block}{c(ccc|cc)}
    \vdots &  \BAmulticolumn{3}{c|}{\multirow{3}{*}{$A(x)$}}&\BAmulticolumn{2}{c}{\multirow{3}{*}{$0$}}\\
    \frac{\partial}{\partial \sigma_{j_i}^{(i)}}&  &&&&\\
    \vdots &  &&&&\\
    \cline{1-6}
    \rule{0pt}{3ex}  
    \frac{\partial}{\partial \tau_{1,2}} &  \BAmulticolumn{3}{c|}{\multirow{3}{*}{$C(x)$}}&\BAmulticolumn{2}{c}{\multirow{3}{*}{$D(x)$}}\\
    \frac{\partial}{\partial \tau_{2,1}}  & &&&&\\
    \frac{\partial}{\partial \tau_{2,2}}  &  &&&&\\
    \end{block}
  \end{blockarray}\,\,\,.
\]
Then, for $x\in S_1$, the matrix $A(x)$ coincide with the Jacobian of $S_1$ at $x$. Thus, we conclude that $x\in S_1$ is a singular point of $C_X$ if and only if the rank of $D(x)$ is not maximal. A similar argument as in Proposition~\ref{prop: smooth loci n-1} shows that the intersection of $S_1$ with the variety defined by the $2\times 2$ minors of $D(x)$ is empty. Hence, we conclude that $C_X$ is smooth at $S_1$.\end{proof}
 
 \para

 \begin{remark}
By \cite[Remark 3.3]{BI}, the maximum number of totally mixed Nash equilibria for a generic $n$-player game with binary choices is not zero. It is in particular the number of derangements of the set $[n]$ \cite[Corollary 6.9]{CBMS}. Let $X$ be a generic game for which there exists a totally mixed Nash equilibrium. Since the Nash CI curve $C_X$ is smooth and contains totally mixed Nash equilibria of $X$, the real points of $C_X$ are Zariski dense (e.g.\ \cite[Theorem 5.1]{sottile16}). As a consequence of Theorem \ref{thm: irreducible Nash CI}, we deduce that for such a game, the intersection of the Nash CI curve with the open simplex is a smooth manifold of dimension $1$. 
 \end{remark}

\subsection{Affine universality}\label{sec: affine universality}
In the Nash case, the notion of universality asks whether every affine real algebraic variety is isomorphic to an affine open subset of the variety $\mathbb{V}(G_1,\ldots,G_n)$ for some payoff tables. In \cite[Theorem 1]{datta}, Datta gives a positive answer to this question. We deal with the analogous question for $C_X$. 

\para
 
Let $U_{X}$ be the intersection of $N_{X}$ with the principal open subset  $D(\tilde{\sigma}^{(1)}_2\cdots\tilde{\sigma}^{(n)}_2)\subset \left(\P^1\right)^n$, i.e.\ the open subset defined by $\tilde{\sigma}^{(i)}_2\neq 0$ for every $i$. 
In \cite{datta}, the notion of isomorphism used is the notion of stable isomorphism in the category of semialgebraic sets. However, we observe that in the proof of \cite[Theorem 6]{datta},
given any real algebraic variety $S$, the author constructs a game $X$ such that the affine open subset $U_X$ and $S$ have isomorphic coordinate rings. Therefore, \cite[Theorem 1,Theorem 6]{datta} can be rephrased as follows using the notion of isomorphism of algebraic varieties.

\para 

\begin{theorem}
Let $S\subset \mathbb{R}^m$ be a real affine algebraic variety. Then, there exists an $n$-player game with binary choices with payoff tables $X^{(1)},\ldots, X^{(n)}$ such that $U_{X}\simeq S$.
\end{theorem}

\para 

Based on the previous theorem, in order to derive analogous results for $C_X$, we  study the relation between $C_X$ and $N_{X}$. Let $W_X$ be the affine open subset of $C_X$ defined as the intersection of $C_X$ with the principal open subset $D(\sigma^{(1)}_2\cdots\sigma^{(n-2)}_2\tau_{2,2})\subseteq \Mc$, i.e.\ the open subset defined by $\sigma^{(1)}_2,\ldots,\sigma^{(n-2)}_2, \tau_{2,2}\neq 0$. Then, we state the following result.

\para 

\begin{proposition}\label{prop:univ 1}
For every $n$-player game with binary choices with payoff tables $\tilde{X}^{(1)},\ldots, \tilde{X}^{(n)}$, there exists an $n+2$-player game with binary choices with payoff tables $X^{(1)},\ldots,X^{(n+2)}$ such that $$W_X\simeq U_{\tilde{X}}\times \mathbb{R}^1.$$
\end{proposition}
\begin{proof}
For every $1\leq i\leq n$, let $\tilde{G}_i\in\mathbb{R}[\tilde{\sigma}_1^{(1)},\ldots,\tilde{\sigma}_1^{(n)}]$ be the polynomial resulting from evaluating $\tilde{\sigma}_2^{(1)},\ldots, \tilde{\sigma}_2^{(n)}$ by $1$ in $G_i$. Then, 
  $U_{\tilde{X}}$ is an affine algebraic subvariety in $\mathbb{R}^n$ defined by $\tilde{G}_1,\ldots,\tilde{G}_n$. Analogously, we consider the variety $C_X=\mathbb{V}(F_1,\ldots,F_{n+2})$ arising from an $n+2$-player game. For every $i$, let $\tilde{F}_i\in\mathbb{R}[\sigma_1^{(1)},\ldots,\sigma_{1}^{(n)},\tau_{1,1},\tau_{1,2},\tau_{2,1}]$ be the polynomial obtained after evaluating  $\sigma_2^{(1)},\ldots,\sigma_{2}^{(n-1)},\tau_{2,2}$ by $1$ in $F_i$. Then, $W_X$ is the subvariety of $\mathbb{R}^{n+2}$ defined by  $\tilde{F}_1,\ldots,\tilde{F}_{n+2}$.

  Now, the idea is to fix payoff tables $X^{(1)},\ldots, X^{(n)}$ such that  $\tilde{F}_1,\ldots,\tilde{F}_{n}$ are $\tilde{G}_1,\ldots,\tilde{G}_{n}$ 
  when substituting $\sigma_1^{(i)}$ by 
  $\tilde{\sigma}_1^{(i)}$ for $i\leq n$.

  First of all, for $i\leq n$ and $j_1,\ldots,j_{n}\in[2]$, we  set
  \[
  X^{(i)}_{j_1\cdots j_{n}11} = X^{(i)}_{j_1\cdots j_{n}12} = X^{(i)}_{j_1\cdots j_{n}21} = 0,
  \]
  and we get that for $i\leq n$ 
  \[
  F_i = \displaystyle\sum_{j_1,\ldots,\widehat{j_i},\ldots,j_{n}}A^{(i)}_{j_1\cdots j_{n}}\sigma^{(1)}_{j_1}\cdots \widehat{\sigma^{(i)}_{j_{i}}}\cdots \sigma^{(n)}_{j_{n}}\tau_{2,2},
  \]
  where  $A^{(i)}_{j_1\cdots j_{n-1}}=X^{(i)}_{j_1\cdots2\cdots j_{n-1}22}- X^{(i)}_{j_1\cdots1\cdots j_{n-1}22}$.
  Then, substituting $\sigma_1^{(i)}$ by $\tilde{\sigma}_1^{(i)}$ in $\tilde{F}_i$, we get a polynomial of the same format as $G_i$. Therefore, we can fix the coefficients $A^{(i)}_{j_1\cdots j_{n}}$ such that this substitution is equal to $\tilde{G}_i$.
  Thus, we obtain that $\mathbb{V}(\tilde{F}_1,\ldots,\tilde{F}_{n})\simeq \mathbb{V}(\tilde{G}_1,\ldots,\tilde{G}_{n})\times \mathbb{R}^3$, where the factor $\mathbb{R}^3$ arises from the factor $\P^3$ of $\Mc$. Note that if  $\tilde{F}_{n+1}$ and $\tilde{F}_{n+2}$ can be chosen as linear forms on $\tau_{1,1},\tau_{1,2}$, and $\tau_{2,1}$, then 
  \[
  \mathbb{V}(\tilde{F}_1,\ldots,\tilde{F}_{n+2})\simeq \mathbb{V}(\tilde{G}_1,\ldots,\tilde{G}_{n})\times \mathbb{V}(\tilde{F}_{n+1},\tilde{F}_{n+2})\simeq \mathbb{V}(\tilde{G}_1,\ldots,\tilde{G}_{n})\times \mathbb{R}^1.
  \]
  
  Thus, our next step is to fix  $X^{(n+1)}$ and $X^{(n+2)}$ such that $\tilde{F}_{n+1}=\tau_{1,1}+\tau_{1,2}$ and $\tilde{F}_{n+2}=\tau_{1,1}+\tau_{2,1}$.  To do so, for every $j_1,\ldots,j_{n+2}$ we set $X_{j_1 \cdots j_{n+2}}^{(n+1)}= X_{j_1\cdots j_{n+2}}^{(n+2)}=0$ if $j_k = 1$ for some $k\leq n$. Then, we get that
  \[
  F_{n+1} = \sigma_{2}^{(1)}\cdots\sigma_{2}^{(n)}
  \left(
  A^{(n+1)}\tau_{1,1}\tau_{2,1}+
    B^{(n+1)}\tau_{1,2}\tau_{2,1}+
      C^{(n+1)}\tau_{1,1}\tau_{2,2}+
      D^{(n+1)}\tau_{1,2}\tau_{2,2}
  \right),
  \]
  where the coefficients $A^{(n+1)},B^{(n+1)},C^{(n+1)}D^{(n+1)}$ are defined as
  \[\begin{array}{ll}
  A^{(n+1)} = X^{(n+1)}_{2\cdots221}-X^{(n+1)}_{2\cdots211}, &
  B^{(n+1)} = X^{(n+1)}_{2\cdots221}-X^{(n+1)}_{2\cdots212},\\  \noalign{\vspace*{1mm}}
  C^{(n+1)} = X^{(n+1)}_{2\cdots222}-X^{(n+1)}_{2\cdots211},& D^{(n+1)} =B^{(n+1)} +C^{(n+1)} -A^{(n+1)}.\end{array} \] Moreover, fixing $A^{(n+1)} =C^{(n+1)} = 0$ and $B^{(n+1)} = 1$, we obtain that 
  \[
  F_{n+1} = \sigma_{2}^{(1)}\cdots\sigma_{2}^{(n)}\tau_{2,2}
  \left(
  \tau_{1,1}+\tau_{1,2}
  \right)
  \]
   and, hence, $\tilde{F}_{n+1}=\tau_{1,1}+\tau_{1,2}$. Similarly, it follows for $F_{n+2}$.
\end{proof}

\para 

Applying the previous proposition in combination with \cite[Theorem 1]{datta}, we derive the following corollary.

\para 

\begin{corollary}\label{co:affine univ 1}
    Let $S$ be a real affine algebraic variety. Then there exists a game with binary choices such that the affine open subset $W_X$ of $C_X$ is isomorphic to $S\times\mathbb{R}^1.$
\end{corollary}

\para 

\begin{remark}
    As a consequence of Corollary \ref{co:affine univ 1}, we deduce that for any singularity type, there exists a game with binary choices such that $C_X$ has that singularity type. This is known as the Murphy's law. Therefore, we deduce that the space of all varieties $C_X$ for binary games with any number of players satisfies the Murphy's Law. For further reading on the Murphy's law in algebraic geometry see \cite{Vakil}.
\end{remark}

\para 

The strategy for Corollary~\ref{co:affine univ 1} was
 to solve the dimension problem of the universality theorem for $C_X$ by adding one to the dimension of the real algebraic variety. Our second approach is to work with real algebraic varieties of dimension at least one. 
 More concretely, we focus on varieties that are cut out by fewer equations than the  dimension of the ambient space. In this situation, the dimension cannot drop to zero.  
 In  the proof of \cite[Theorem 6]{datta}, the author exploits the structure of the polynomials $G_i$ in equation \eqref{eq:nash poly} defining the variety $N_X$. In our setting, 
 the polynomials $F_1,\ldots,F_n$
 defining $C_X$ are slightly different to the ones defining $N_X$ due to 
 the factor $\P^3$ of the Segre variety $\mathcal{M}_\mathcal{C}$. In the following result, we slightly modify the proof of \cite[Theorem 6]{datta} to be compatible with our system of equations defining $C_X$.

\para 

\begin{theorem}\label{thm: universality of Spohn CI}
Let $S\subseteq \mathbb{R}^n$ be a real affine algebraic variety defined by  $G_1,\ldots,G_m\in\mathbb{R}[x_1,\ldots,x_n]$ with $m<n$. For every $i\in\{1,\ldots,n\}$, let $\delta_{i}$
be 
the maximum of the degrees of $x_i$ in $G_1,\ldots,G_m$.
Then, there exists a  $(\delta+n+1)-$player game with binary choices such that the affine open subset $W_X$ of $C_X$ is isomorphic to $S$, where $\delta = \delta_1+\cdots +\delta_n$.
\end{theorem}

\begin{proof}
First of all, note that for a $(\delta+n+1)-$player game, $C_X\subset \left(\P^1\right)^{\delta}\times\left(\P^1\right)^{n-1}\times \P^3$. We denote the first $\delta$ variables by $\sigma_1^{(i,j_i)}$ for $i\in\{1,\ldots,n\}$ and $j_i\in \{1,\ldots , \delta_i\}$. Let $\sigma_1^{(i)}$ for $i\in\{1,\ldots,n-1\}$ and $\tau_{1,1},\tau_{1,2},\tau_{2,1},\tau_{2,2}$ be the variables of the rest factors of the Segre variety. Since we are restricting our study to the affine open subset $W_X$, we can assume that $\sigma_2^{(i,j)},\sigma_2^{(k)},$ and $\tau_{2,2}$ are $1$.  Let $\tilde{F}_{1,1},\ldots,\tilde{F}_{n,\delta_n},\tilde{F}_1,\ldots,\tilde{F}_{n+1}$  be the polynomials defining $W_X$. 
	Now, we fix the payoff tables of the game such that the first $\delta$ generators of $W_X$ are
	\[
	\begin{array}{cccc}
	\tilde{F}_{1,1}=\sigma_1^{(1,1)}-\sigma_1^{(1)}, & \tilde{F}_{1,2}=\sigma_1^{(1,2)}-\sigma_1^{(1)} \sigma_1^{(1,1)}, & \cdots &, \tilde{F}_{1,\delta_1}=\sigma_1^{(1,\delta_1)}-\sigma_1^{(1)} \sigma_1^{(1,\delta_1-1)},\\
	\vdots&\vdots & &\vdots\\
	\tilde{F}_{n,1}=\sigma_1^{(n,1)}-\tau_{1,1}, & \tilde{F}_{n,2}=\sigma_1^{(n,2)}-\tau_{1,1} \sigma_1^{(n,1)}, & \cdots &, \tilde{F}_{n,\delta_n}=\sigma_1^{(n,\delta_n)}-\tau_{1,1} \sigma_1^{(n,\delta_n-1)}.
	\end{array}
	\]
	This implies that in the coordinate ring of $W_X$, we have that $\sigma_1^{(i,j_i)}=(\sigma_1^{(i)})^{\,j_i}$ for $1\leq i\leq n-1$ and  $\sigma_1^{(n,j_i)} = \tau_{1,1}^{j_i}$. 
	In addition, for $i\leq m$, we set $\tilde{F}_i$ to be the polynomial obtained from replacing $x_i^k$ by $\sigma_1^{(i,k)}$ in $G_i$. Since $m<n$, $\tilde{F}_n$ and $\tilde{F}_{n+1}$ are not among the generators we have already set up. 
 Note that the main variations of this proof from the one of \cite[Theorem 6]{datta} arise 
 while dealing with the polynomials $\tilde{F}_n$ and $\tilde{F}_{n+1}$ and the variables $\tau_{i,j}$. 
 As in the proof of Proposition~\ref{prop:univ 1}, we set  $\tilde{F}_n=\tau_{1,1}+\tau_{1,2}$ and $\tilde{F}_{n+1}=\tau_{1,1}+\tau_{2,1}$. Finally, we fix  $\tilde{F}_{k}=0$ for $m<k< n$.
	 Now that we have fixed the ideal of $W_X$, we construct an isomorphism between  the coordinate rings of $S$ and $W_X$, denoted by $R_1$ and $R_2$ respectively. We consider the ring map from $R_1$ to $R_2$ that sends $x_1,\ldots,x_n$ to $\sigma_1^{(1)},\ldots,\sigma_1^{(n-1)},\tau_{1,1}$ respectively. Then, by construction, $G_i$ is mapped to $\tilde{F}_i$, and thus, the map is well--defined. Moreover, one can check that this map has an inverse map that  sends $\sigma_1^{(i,j_i)}$ to $x_i^{j_i}$, $\tau_{1,1}$ to $x_n$, and $\tau_{1,2}$ and $\tau_{2,1}$ to $-x_n$. Therefore, we conclude that
 $W_X =\mathbb{V}(\tilde{F}_{1,1},\ldots,\tilde{F}_{n,\delta_n},\tilde{F}_{1},\ldots,\tilde{F}_{n})$ is isomorphic to $S$. Finally, we need to ensure that $W_X$ is an affine open subset of the Spohn CI variety of a game. We observe that the variable $\sigma_1^{(i,j)}$ appears in the polynomial $\tilde{F}_{i,j}$, which should not happen if $W_X$ is constructed from a game. To solve this issue, we need to relabel the polynomials $\tilde{F}_{1,1},\ldots, \tilde{F}_{n,\delta_n}$ such that  $\sigma_1^{(i,j)}$ does not appear in the polynomial $\tilde{F}_{i,j}$.
		  The existence of such relabelling is proven using the same argument as in the proof of \cite[Theorem 6]{datta}. 
\end{proof}

\begin{remark}
Given a real affine algebraic variety $S\subseteq \mathbb{R}^n$ defined by  $G_1,\ldots,G_m\in\mathbb{R}[x_1,\ldots,x_n]$ with $m<n$,
 Theorem~\ref{thm: universality of Spohn CI} provides an effective method of computing a game $X$ such that $U_{X}\simeq S$. 
\end{remark}

\para

\begin{example}
Consider the real plane curve defined by $G_1= x_1^2+x_2^2-1$. By Theorem~\ref{thm: universality of Spohn CI}, this plane curve can be described by a 7-player game with binary choices. Consider the polynomials
\[\begin{array}{l}

F_{1} =\sigma_2^{(2)}\sigma_2^{(4)}\sigma_2^{(5)}\left( \sigma_1^{(3)}\tau_{2,2}-\tau_{1,1}\sigma_2^{(3)}\right),\\

F_{2} =\sigma_2^{(1)}\sigma_2^{(5)}\left( \sigma_2^{(3)}\sigma_1^{(4)}\tau_{2,2}-\sigma_1^{(3)}\sigma_2^{(4)}\tau_{1,1}\right),\\

F_{3} =\sigma_2^{(2)}\sigma_2^{(4)}\tau_{2,2}\left( \sigma_1^{(1)}\sigma_2^{(5)}-\sigma_1^{(5)}\sigma_2^{(1)}\right),\\

F_{4} =\sigma_2^{(3)}\tau_{2,2}\left( \sigma_2^{(1)}\sigma_1^{(2)}\sigma_2^{(5)}-\sigma_1^{(1)}\sigma_2^{(2)}\sigma_1^{(5)}\right),\\

F_5 = \sigma_2^{(1)}\sigma_2^{(3)}\tau_{2,2}\left(\sigma_2^{(4)}\sigma_1^{(2)}+\sigma_2^{(2)}\sigma_1^{(4)}-\sigma_2^{(2)}\sigma_2^{(4)}\right),\\

F_6 =\sigma_2^{(1)}\sigma_2^{(2)}\sigma_2^{(3)}\sigma_2^{(4)}\sigma_2^{(5)}\tau_{2,2}( \tau_{1,1}+\tau_{1,2}),\\

F_7 =\sigma_2^{(1)}\sigma_2^{(2)}\sigma_2^{(3)}\sigma_2^{(4)}\sigma_2^{(5)}\tau_{2,2}( \tau_{1,1}+\tau_{2,1}).

\end{array}
\]
One can check that the open subset of $\mathbb{V}(F_1,\ldots,F_7)$ defined by $\sigma_2^{(1)}\cdots\sigma_2^{(5)}\tau_{2,2}\neq 0 $ is isomorphic to $\mathbb{V}(G_1)$. Moreover, the tuple $(F_1,\ldots,F_7)$ lies in the image of the linear map \eqref{eq:surj 2}. Computing its preimage, one can compute a linear subspace of games $X$ such that $C_X$ is defined by the equations $F_{1},\ldots,F_7$  and the open subset $U_X$ is isomorphic to $\mathbb{V}(G_1)$. For example, the possible payoff tables corresponding to the players $1$, $5$, $6$, and $7$ are 
\[
\begin{alignedat}{2}
X^{(1)}_{i_{1}\cdots i_7 } &= \left\{ \begin{array}{cl}
1 &   \text{if } (i_{1},\ldots,i_7)\in\{(2,2,1,2,2,2,2),(1,2,2,2,1,1,1)\},
\\
0&\text{else.}
\end{array}\right.
\\
X^{(5)}_{i_{1}\cdots i_7 } &= \left\{ 
\begin{array}{cl}
\!1 &   \text{if } (i_{1},\ldots,i_7)\in\{(2,1,2,2,2,2,2),(2,2,2,1,2,2,2),(2,2,2,2,1,2,2)\},
\\
\!0&\text{else.}
\end{array}\right.
\\
X^{(6)}_{i_{1}\cdots i_7 } &=X^{(7)}_{i_{1}\cdots i_7 } = \left\{ 
\begin{array}{cl}
1 &   \text{if } (i_{1},\ldots,i_7)\in\{(2,2,2,2,2,2,2)\},
\\
0&\text{else.}
\end{array}\right. 
\end{alignedat} 
\]

\end{example}

\begin{remark}
    In \cite{datta}, Datta's universality theorem refers to  the set of totally mixed Nash equilibria. An analogous statement for the set of totally mixed CI equilibria can be obtained in our setting. Namely, given a real affine algebraic variety $S$, there exists a game with binary choices such that $C_X\cap \Delta $ is isomorphic to $S\times \mathbb{R}^1$ (Corollary \ref{co:affine univ 1}). As in \cite{datta}, here we use the notion of stable isomorphism in the category of semialgebraic sets. To derive these results one should argue as in \cite{datta}: the set of real points of a real affine algebraic variety is isomorphic to the set of real points of a real affine algebraic variety whose real points are contained in the probability simplex. Now, assuming the latter, the statement follows from  Proposition \ref{prop:univ 1}.
     An analogous statement also holds for Theorem \ref{thm: universality of Spohn CI}.
\end{remark}

\subsection*{Acknowledgements}
We are grateful to Daniele Agostini and Bernd Sturmfels for many helpful and inspiring discussions on the topic. Javier Sendra--Arranz received the support of a fellowship from the "la Caixa" Foundation (ID 100010434). The fellowship code is LCF/BQ/EU21/11890110.

\addcontentsline{toc}{section}{References}\label{sec:references}
\bibliographystyle{plain}

\noindent Irem Portakal, Technical University of Munich
\hfill {\tt mail@irem-portakal.de}

\noindent Javier Sendra--Arranz, MPI-MiS Leipzig
\hfill {\tt javier.sendra@mis.mpg.de}
\end{document}